\documentclass[11pt]{article}
\usepackage[margin=1in]{geometry}
\usepackage{xcolor}
\usepackage{graphicx, float}
\graphicspath{ {./Images/} }
\usepackage{booktabs}
\usepackage{amsmath}
\usepackage{amssymb}
\usepackage{amsthm}
\usepackage{mathtools}
\usepackage{makecell}
\usepackage{comment}
\usepackage{tikz}
\usepackage{hyperref}
\usepackage{mathtools}

\usepackage{tikz-network}

\newcommand{\N}{\mathbb{N}}

\newtheorem{theorem}{Theorem}

\newtheorem{conjecture}[theorem]{Conjecture}

\newtheorem{definition}[theorem]{Definition}

\newtheorem{lemma}[theorem]{Lemma}

\newtheorem{proposition}[theorem]{Proposition}
\newtheorem{remark}[theorem]{Remark}

\usepackage[parfill]{parskip}

\title{THE PLANAR TURÁN NUMBER OF $\Theta_6$-GRAPHS} % The short title appears at the bottom of every slide, the full title is only on the title page

\author{David Guan, Ervin Gy\H{o}ri, Diep Luong-Le, Felicia Wang, Mengyuan Yang} % Your name
\medskip
\date{\today}
\begin{document}
\maketitle

\begin{abstract}
   There are two particular $\Theta_6$-graphs - the 6-cycle graphs with a diagonal. We find the planar Tur\'an number of each of them, i.e. the maximum number of edges in a planar graph $G$ of $n$ vertices not containing the given  $\Theta_6$ as a subgraph and we find infinitely many extremal constructions showing the sharpness of these results - apart from a small additive constant error in one of the cases. 
   
   %The proofs are mostly new applications or variants of the “contribution method” introduced by Ghosh, Gy\H ori, Martin, Paulos and Xiao.
\end{abstract}

\section{Introduction and Main Results}
All graphs considered in this paper are finite, undirected, and simple plane graph. We denote the vertex, edge, and face sets of a graph $G$ by $V(G)$, $E(G)$, and $F(G)$, respectively.
Moreover, we denote the edge set of some $f\in F(G)$ by $E(f)$.

A classical problem in extremal graph theory is to determine the Tur\'an number $ex(n,H)$, which is the maximum number of edges in a graph of $n$ vertices that does not contain $H$ as a subgraph. In 2016, Dowden initiated the study of planar Tur\'an numbers, $ex_{\mathcal{P}}(n,H)$, which is the maximum number of edges in a planar graph of $n$ vertices not containing $H$ as a subgraph. 
Let $\Theta_k$ denote the family of Theta graphs on $k \ge 4$ vertices, that is, graphs obtained by joining a pair of non-consecutive vertices of a cycle $C_k$ with an edge.
Dowden  and other authors  (\cite{Dowden}, \cite{theta4}, \cite{c6}, \cite{H-free}, \cite{thetafree}) determined the planar Tur\'an numbers of $C_4$, $C_5$, $C_6$ $\Theta_4$, $\Theta_5$, $\Theta_6$ and more.

In \cite{theta6}, Gy\H ori et. al. proved the following theorem for the planar Tur\'an number of $\Theta_6$.
\begin{theorem} $ex_{\mathcal{P}}(n,\Theta_6)\leq \frac{18}{7}n-\frac{48}{7}$ for all $n\geq 14$.
\end{theorem}

\noindent They used the triangular block contribution method to estimate this planar Tur\'an number.
They conjectured the planar Tur\'an numbers for the particular $\Theta_6$ graphs.
\begin{conjecture}(\cite{theta6})\hfill
\begin{enumerate}
    \item $ex_{\mathcal{P}}(n,\Theta^1_6)=\frac{45}{17}n+O(1)$.
    \item $ex_{\mathcal{P}}(n,\Theta^2_6)=\frac{18}{7}n+O(1)$.
\end{enumerate}
\end{conjecture}
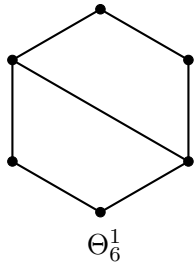
\begin{center}
\begin{tabular}{c c}
\begin{tikzpicture}[scale=0.9]
\filldraw (0,1.5) circle (0.07);
\filldraw (0,-1.5) circle (0.07);
\filldraw (-1.3,0.75) circle (0.07);
\filldraw (1.3,0.75) circle (0.07);
\filldraw (-1.3,-0.75) circle (0.07);
\filldraw (1.3,-0.75) circle (0.07);
\draw[thick] (0,1.5) -- (-1.3,0.75);
\draw[thick] (-1.3,-0.75) -- (-1.3,0.75);
\draw[thick] (-1.3,-0.75) -- (0,-1.5);
\draw[thick] (0,1.5) -- (1.3,0.75);
\draw[thick] (1.3,-0.75) -- (1.3,0.75);
\draw[thick] (1.3,-0.75) -- (0,-1.5);
\draw[thick] (-1.3,0.75) -- (1.3,-0.75);
\end{tikzpicture}
\hspace{1cm}
& 
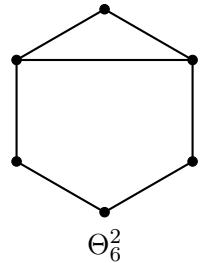
\begin{tikzpicture}[scale=0.9]
\filldraw (0,1.5) circle (0.07);
\filldraw (0,-1.5) circle (0.07);
\filldraw (-1.3,0.75) circle (0.07);
\filldraw (1.3,0.75) circle (0.07);
\filldraw (-1.3,-0.75) circle (0.07);
\filldraw (1.3,-0.75) circle (0.07);
\draw[thick] (0,1.5) -- (-1.3,0.75);
\draw[thick] (-1.3,-0.75) -- (-1.3,0.75);
\draw[thick] (-1.3,-0.75) -- (0,-1.5);
\draw[thick] (0,1.5) -- (1.3,0.75);
\draw[thick] (1.3,-0.75) -- (1.3,0.75);
\draw[thick] (1.3,-0.75) -- (0,-1.5);
\draw[thick] (-1.3,0.75) -- (1.3,0.75);
\end{tikzpicture}\\
$\Theta_6^1$ \hspace{0.9cm}
& $\Theta_6^2$  
\end{tabular}
\end{center}
In this paper, we use that technique and obtain the following theorems.

\begin{theorem} \label{theta6-1}
    $ex_{\mathcal{P}}(n,\Theta^1_6)\leq\frac{45}{17}(n-2)$ for all $n\geq 6$, and the equality holds for infinitely many $n$.
\end{theorem}

\begin{theorem} \label{theta6-2}
    $ex_{\mathcal{P}}(n,\Theta^2_6)\leq\frac{18}{7}(n-2)$ for all $n\geq 6$, and $ex_{\mathcal{P}}(n,\Theta^2_6)\geq\frac{18}{7}n-\frac{48}{7}$ for infinitely many $n$.
\end{theorem}

\section{Triangular Blocks}
In this section, we will introduce our main proof technique in a general context.
% Let $H$ be some graph and $G$ be an $n$-vertex $H$-free planar graph.
To be precise, we will divide a plane graph $G$ into pieces of the following form.

\begin{definition}\label{tri_block}
Let $G$ be a plane graph. 
Pick $e\in E(G)$.
If $e$ is not adjacent to any $3$-face of $G$, then we call the subgraph of $G$ induced by $e$ a \textbf{trivial triangular block}, denoted by $B_e$.
If $e$ is incident to some $3$-face $f_e\in F(G)$, then we find $\Omega_e\subseteq F(G)$ with the smallest $|\Omega_e |$ such that
\begin{enumerate}
    \item $f_e\in\Omega_e$;
    \item if $f'\in F(G)$ is a $3$-face adjacent to some $f''\in \Omega_e$, then $f'\in \Omega_e$.
\end{enumerate}
Let $B_e$ be the subgraph of $G$ induced by $\bigcup_{f\in\Omega_e}E(f)$. 
We call $B_e$ a \textbf{triangular block}, and every $f_e\in\Omega_e$ an \textbf{interior face} of $B_e$.
\end{definition}

\begin{remark}\label{adj}
    Every $f\in F(G)$ adjacent to $B$ that is not an interior face of $B$ cannot be a $3$-face without violating the second axiom in Definition \ref{tri_block}.
\end{remark}

% \begin{definition}
% Let $G$ be a finite planar graph and fix an edge $e\in E(G)$.
% Let $B_e$ be the connected
% subgraph of $G$ with the smallest $\left|E(B_e)\right|$ such that 
% \begin{enumerate}
% \item $e\in E(B_e)$;
% \item every $3$-face of $G$ is either contained in $B_e$ or edge-disjoint from
% $B_e$.
% \end{enumerate}
% Then let $\Omega$ be the set of $3$-faces of $G$ that are also $3$-faces of $B_e$.
% We call $(B_e, \Omega)$ the \textbf{triangular block containing $e$}.
% Moreover, we call $B_e$ the \textbf{skeleton} and every $f\in\Omega$ an \textbf{interior face} of this triangular block.
% If $B_e$ only contains a single edge $e$,
% then $(B_e,\Omega=\emptyset)$ is called a \textbf{trivial triangular
% block}. 
% We use only $B_e$, or simply $B$, to denote a triangular block when there is no ambiguity in $\Omega$.
% In general, we say $B$ is a \textbf{triangular block} if it is a triangular block containing $e$ for some edge $e \in E(G)$.
% \end{definition}

% \begin{definition}\label{recursive}
%     Let $G$ be a plane graph. For an edge $e \in E(G)$, if it’s not contained in any of the $3$-faces of $G$, then we call it a \textbf{trivial triangular block}. Otherwise, we perform the following algorithm to construct a \textbf{triangular block $B_e$ containing $e$}. \\ \\
%     $ B_e \leftarrow (V (e), e);$\\
%     \textbf{while} there exists an edge in $B_e$ such that it is in  $3$-face of $G$ which is not contained in $B_e$ \textbf{do} \\
%     \ \ \ add all the edges of such 3-faces to $B_e$; 
%     \textbf{end}\\
%     Output $B_e$.
% \end{definition}

Intuitively, we can construct a triangular block as follows:
\begin{enumerate}
    \item Add a $3$-face to $\Omega$;
    \item Find a $3$-face adjacent to some face in $\Omega$ and add it to $\Omega$;
    \item Repeat the previous step and stop when we cannot add $3$-faces anymore.
\end{enumerate}
Then the triangular block is simply induced by the boundaries of the faces in $\Omega$.
Notice that, by construction, faces in $\Omega$ are all $3$-faces and they are connected through sharing edges.
This connectedness implies that triangular blocks are well-defined, that is, our definition always produces the same block for the same $e\in E(G)$.
More importantly, triangular blocks form a nice division of $G$:
% Here are two facts about the $\Omega$'s.
% First, since faces in $\Omega$ are connected through edges, we can start with any face in $\Omega$ and obtain the same set of interior faces.
% Second, if $\Omega$ and $\Omega'$ are two distinct sets of interior faces, by the second property in Definition \ref{tri_block}, we know $f\in \Omega$ and $f'\in\Omega'$ must be edge-disjoint.
% One can use these two facts to prove the following proposition.
\begin{proposition}\label{blocksequivrelationprop}
For $e, e' \in E(G)$, we define $e\sim e'$ if they are contained in the same (possibly trivial) triangular block.
Then $\sim$ is an equivalence relation.
\end{proposition}
This proposition implies that there is a unique triangular block containing each $e\in E(G)$.
Let $\mathcal{B}$ be the set of (trivial) triangular blocks of $G$.
We thus know $\mathcal{B}$ forms a partition of $E(G)$. 

% \begin{proof}
% Clearly $B_e = B_{e'}$ implies both $e \in B_{e'}$ and $e' \in B_e$. Now, suppose $e' \in B_e$. Applying the recursive definition, this implies that every element in $B_{e'}$ can be generated by $e$, thus $B_{e'} \subseteq B_e$. Also, we know that every edge $\alpha \in B_e$ is generated by some edge sharing a $3$-face with $\alpha$. So there is a finite sequence $e =: e_0, e_1, ..., e_{k-1}, e_k := e'$ such that all $e_i$'s are distinct and $e_i, e_{i+1}$ shares a $3$-face for every $i$. By applying the recursive definition backwards, one can get $e \in B_{e'}$. So $B_e \subseteq B_{e'}$ and that $B_e = B_e'$. 
% \end{proof}

% \begin{proof}
% This immediately follows from proposition \ref{blocksequivrelationprop} that two triangular blocks are either the same or edge-disjoint. 
% \end{proof}

We will next give the full list of the triangular blocks appearing in $\Theta_6$-free graphs.
An easy (but somewhat tedious) case-by-case analysis yields the following proposition and lemma.
\begin{proposition} \label{8blocks}
%Any triangular block with at most $5$ vertices is isomorphic to one of the graphs below. 
The only triangular blocks with at most $5$ vertices are as follows.
\end{proposition}
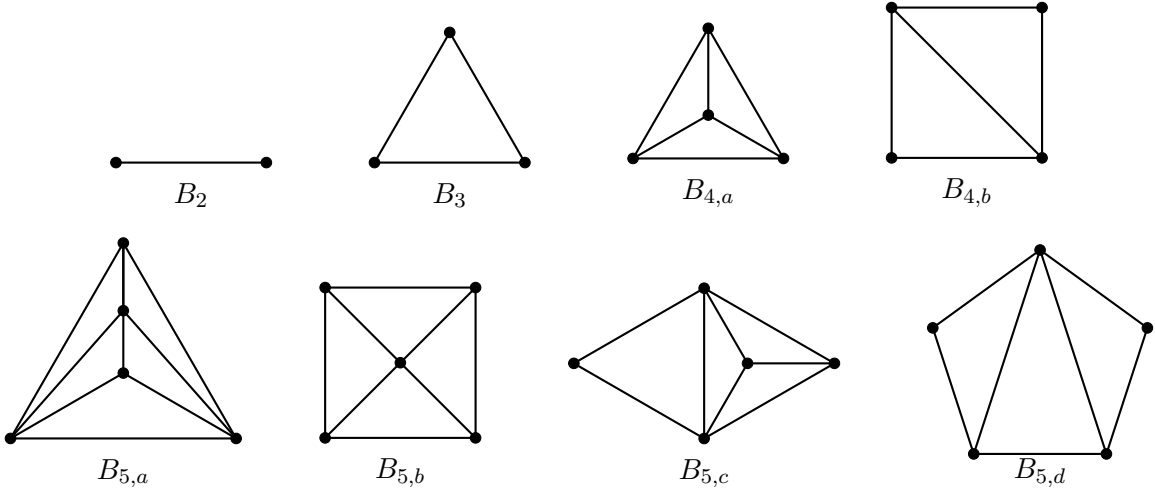
\begin{center}
    \begin{tikzpicture}
    \filldraw (-1,0) circle (0.07);
    \filldraw (1,0) circle (0.07);
    \draw[thick] (-1,0) -- (1,0);
    \node at (0,0) [label=below: $B_2$]{};
    \end{tikzpicture} \quad \quad \quad
    \begin{tikzpicture}
    \filldraw (-1,0) circle (0.07);
    \filldraw (1,0) circle (0.07);
    \filldraw (0,1.732) circle (0.07);
    \draw[thick] (-1,0) -- (1,0);
    \draw[thick] (-1,0) -- (0,1.732);
    \draw[thick] (0,1.732) -- (1,0);
    \node at (0,0) [label=below: $B_3$]{};
    \end{tikzpicture} \quad \quad \quad 
    \begin{tikzpicture}
    \filldraw (-1,0) circle (0.07);
    \filldraw (1,0) circle (0.07);
    \filldraw (0,1.732) circle (0.07);
    \filldraw (0,0.577) circle (0.07);
    \draw[thick] (-1,0) -- (1,0);
    \draw[thick] (-1,0) -- (0,1.732);
    \draw[thick] (0,1.732) -- (1,0);
    \draw[thick] (-1,0) -- (0,0.577);
    \draw[thick] (1,0) -- (0,0.577);
    \draw[thick] (0,1.732) -- (0,0.577);
    \node at (0,0) [label=below: $B_{4,a}$]{};
    \end{tikzpicture} \quad \quad \quad 
    \begin{tikzpicture}
    \filldraw (-1,1) circle (0.07);
    \filldraw (1,1) circle (0.07);
    \filldraw (1,-1) circle (0.07);
    \filldraw (-1,-1) circle (0.07);
    \draw[thick] (-1,1) -- (1,1);
    \draw[thick] (-1,1) -- (-1,-1);
    \draw[thick] (1,-1) -- (1,1);
    \draw[thick] (1,-1) -- (-1,-1);
    \draw[thick] (-1,1) -- (1,-1);
    \node at (0,-1) [label=below: $B_{4,b}$]{};
    \end{tikzpicture}
\end{center}
\begin{center}
    \begin{tikzpicture}
    \filldraw (-1.5,0) circle (0.07);
    \filldraw (1.5,0) circle (0.07);
    \filldraw (0,2.6) circle (0.07);
    \filldraw (0,0.87) circle (0.07);
    \filldraw (0,1.7) circle (0.07);
    \draw[thick] (-1.5,0) -- (1.5,0);
    \draw[thick] (-1.5,0) -- (0,2.6);
    \draw[thick] (0,2.6) -- (1.5,0);
    \draw[thick] (-1.5,0) -- (0,0.87);
    \draw[thick] (1.5,0) -- (0,0.87);
    \draw[thick] (0,2.6) -- (0,0.87);
    \draw[thick] (-1.5,0) -- (0,1.7);
    \draw[thick] (1.5,0) -- (0,1.7);
    \draw[thick] (0,2.6) -- (0,1.7);
    \node at (0,0) [label=below: $B_{5,a}$]{};

    \end{tikzpicture}\quad \quad \ \
    \begin{tikzpicture}
    \filldraw (-1,1) circle (0.07);
    \filldraw (1,1) circle (0.07);
    \filldraw (1,-1) circle (0.07);
    \filldraw (-1,-1) circle (0.07);
    \filldraw (0,0) circle (0.07);
    \draw[thick] (-1,1) -- (1,1);
    \draw[thick] (-1,1) -- (-1,-1);
    \draw[thick] (1,-1) -- (1,1);
    \draw[thick] (1,-1) -- (-1,-1);
    \draw[thick] (1,1) -- (0,0);
    \draw[thick] (-1,-1) -- (0,0);
    \draw[thick] (-1,1) -- (0,0);
    \draw[thick] (1,-1) -- (0,0);
    \node at (0,-1) [label=below: $B_{5,b}$]{};
    \end{tikzpicture}  \quad \quad \ \ 
    \begin{tikzpicture}
    \filldraw (0,1) circle (0.07);
    \filldraw (0,-1) circle (0.07);
    \filldraw (-1.732,0) circle (0.07);
    \filldraw (1.732,0) circle (0.07);
    \filldraw (0.577,0) circle (0.07);
    \draw[thick] (0,1) -- (0,-1);
    \draw[thick] (0,1) -- (-1.732,0);
    \draw[thick] (0,1) -- (1.732,0);
    \draw[thick] (0,-1) -- (-1.732,0);
    \draw[thick] (0,-1) -- (1.732,0);
    \draw[thick] (0,1) -- (0.577,0);
    \draw[thick] (0,-1) -- (0.577,0);
    \draw[thick] (1.732,0) -- (0.577,0);
    \node at (0,-1) [label=below: $B_{5,c}$]{}; 
    \end{tikzpicture} \quad \quad \ \ 
    \begin{tikzpicture}
    \filldraw (0,1.5) circle (0.07);
    \filldraw (0.951*1.5,0.309*1.5) circle (0.07);
    \filldraw (-0.951*1.5,0.309*1.5) circle (0.07);
    \filldraw (0.588*1.5,-0.809*1.5) circle (0.07);
    \filldraw (-0.588*1.5,-0.809*1.5) circle (0.07);
    \draw[thick] (0.951*1.5,0.309*1.5) -- (0,1.5);
    \draw[thick] (0.951*1.5,0.309*1.5) -- (0.588*1.5,-0.809*1.5);
    \draw[thick] (-0.588*1.5,-0.809*1.5) -- (0.588*1.5,-0.809*1.5);
    \draw[thick] (-0.588*1.5,-0.809*1.5) -- (-0.951*1.5,0.309*1.5);
    \draw[thick] (0,1.5) -- (-0.951*1.5,0.309*1.5);
    \draw[thick] (0,1.5) -- (0.588*1.5,-0.809*1.5);
    \draw[thick] (0,1.5) -- (-0.588*1.5,-0.809*1.5);
    \node at (0,-1) [label=below: $B_{5,d}$]{}; 
    \end{tikzpicture}    
\end{center}

%\begin{proof}
%Let $B$ be a triangular block on $n$ vertices.
%By Definition \ref{recursive}, we may assume that $B$ is obtained by joining an additional vertex to a triangular block on $n-1$ vertices.
%This assumption implies an inductive way to construct triangular blocks.
%If $n=2$, then $B$ is the trivial triangular block.
%If $n=3$, then we can obtain $B$ by joining an additional vertex to $B_2$, which is either a path of length $2$ or a triangle.
%The former is not a triangular block.
%Applying similar arguments to $n=4,5$ completes the proof.
%\end{proof}

%\begin{definition} i don't think we ever use this
 %   Let $G$ be a plane graph. A vertex $v$ in G is called a junction vertex if it is shared by at least two triangular-blocks of $G$.
%\end{definition}

\begin{lemma}\label{poss_tri_block}
Let $B$ be a triangular block of some plane graph.
If $B$ is $\Theta_6^1$-free, then $B$ either contains at most $5$ vertices or is the following graph.
If $B$ is $\Theta_6^2$-free, then $B$ contains at most $5$ vertices.
\end{lemma}
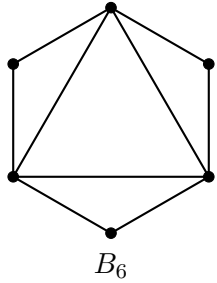
\begin{center}
\begin{tikzpicture}
\filldraw (0,1.5) circle (0.07);
\filldraw (0,-1.5) circle (0.07);
\filldraw (-1.3,0.75) circle (0.07);
\filldraw (1.3,0.75) circle (0.07);
\filldraw (-1.3,-0.75) circle (0.07);
\filldraw (1.3,-0.75) circle (0.07);
\draw[thick] (0,1.5) -- (-1.3,0.75);
\draw[thick] (-1.3,-0.75) -- (-1.3,0.75);
\draw[thick] (-1.3,-0.75) -- (0,-1.5);
\draw[thick] (0,1.5) -- (1.3,0.75);
\draw[thick] (1.3,-0.75) -- (1.3,0.75);
\draw[thick] (1.3,-0.75) -- (0,-1.5);
\draw[thick] (0,1.5) -- (-1.3,-0.75);
\draw[thick] (-1.3,-0.75) -- (1.3,-0.75);
\draw[thick] (1.3,-0.75) -- (0,1.5);
\node at (0,-1.5) [label=below: $B_6$]{}; 
\end{tikzpicture}
\end{center}

Because of this lemma, we do not need to consider other triangular blocks in this paper.

We will now define the \textbf{contribution} of a (trivial) triangular block to the number of edges and faces of the plane graph $G$, that is, we will ``evenly distribute" $e(G)$ and $f(G)$ to its (trivial) triangular blocks.
Here are the exact definitions.

\begin{definition}\label{face_edge_number}

Let $G$ be a plane graph and $B$ be a (possibly trivial) triangular block of $G$.  
We denote its contribution to the total number of edges by
\begin{align*}
e_B = |E(B)|
\end{align*}
and its contribution to the total number of faces by
\begin{align*}
f_B = \sum_{f \in F(G)} \frac{|E(f) \cap E(B)|}{|E(f)|}.
\end{align*}

\end{definition}

If $f\in F(G)$ is not adjacent to $B$, then the corresponding summand in $f_B$ is $0$;
if $f\in F(G)$ is adjacent to $B$, then the corresponding summand in $f_B$ is the proportion of the boundaries of $f$ contained in $B$.
In particular, if $f\in F(G)$ is an interior face of $B$, then the summand is simply $1$.
Since only faces adjacent to $B$ will have non-zero contribution, we only need to consider adjacent faces when calculating $f_B$.

%\begin{remark}
%Since the (weighted) face number $f_B$ and edge number $e_B$ are defined, it is natural to define another important property, the (weighted) vertex number $v_B$, analogous to the definition of $f_B$. However, $v_B$ is never used later in this paper, making this definition redundant and unnecessary. 
%\end{remark}
Since $\mathcal{B}$ forms a partition of $E(G)$,
we obtain an identity:
\begin{equation}\label{equ-1}
\sum_{B\in\mathcal{B}} e_B =e(G).
\end{equation}
Locally, $\mathcal{B}$ also forms a partition of $E(f)$ for each $f\in F(G)$, which gives a second identity:
\begin{equation}\label{equ-2}
\sum_{B \in \mathcal{B}} f_B 
= \sum_{B \in \mathcal{B}} \sum_{f \in F(G)} \frac{|E(f) \cap E(B)|}{|E(f)|} 
= \sum_{f \in F(G)}  \sum_{B \in \mathcal{B}}\frac{|E(f) \cap E(B)|}{|E(f)|}
= \sum_{f \in F(G)} 1
=f(G).
\end{equation}

In the last part of this section, we will convert the inequalities in our theorems to ``small" inequalities involving only the contributions of each (trivial) triangular block.
Proving the theorems are thus reduced to proving each of these ``small" inequality.

Let $H$ be some forbidden graph.
We want to prove some linear bound of the form $ex_{\mathcal{P}}(n,H)\leq \frac{\alpha}{\beta}(n-2)$, called \textbf{target inequality}, for $\alpha,\beta\in \N$.
By definition, it suffices to show that for any $H$-free plane graph $G$, we have $e(G)\leq \frac{\alpha}{\beta}(n(G)-2)$.
Applying Euler's formula, it's equivalent to show $e(G)\leq \frac{\alpha}{\beta}(e(G)-f(G))$ or $(\beta-\alpha)e(G)+\alpha f(G)\leq 0$.
Define $g(x,y)=(\beta-\alpha)x+\alpha y$ to be the \textbf{contribution formula}.
By Equation (\ref{equ-1}) and Equation (\ref{equ-2}),
the original problem is transformed into the following lemma.
\begin{lemma}
If $g(e_{B},f_{B})\leq 0$ for all $B\in \mathcal{B}$ then $ex_{\mathcal{P}}(n,H)\leq \frac{\alpha}{\beta}(n-2)$.
In particular, if the equality holds for some $n$, then every $g(e_{B},f_{B})= 0$ in the extremal construction.
\end{lemma}

However, this lemma is sometimes not true.
It is possible that the contribution of some $B$ is positive, but when considered with possible adjacent triangular blocks, we show that the group as a whole will have non-positive contribution. Thus, we introduce the following definition.

Let $C_1,C_2,\dots,C_k$ be a partition of $\mathcal{B}$.
We can extend Definition \ref{face_edge_number} as follows:
\begin{align*}
    e_{C_i}=\sum_{B\in C_i}e_B
    \text{ and }
    f_{C_i}=\sum_{B\in C_i}f_B.
\end{align*}
Here, we call $C_i$ a \textbf{cluster} of triangular blocks and call $e_{C_i}$ and $f_{C_i}$ the contribution of $C_i$ to the edge number and face number, respectively.
Since $C_i$ forms a partition of $\mathcal{B}$, analogues of Equation (\ref{equ-1}) and Equation (\ref{equ-2}) hold:
\begin{equation}
\sum_{i=1}^k e_{C_i} =e(G)
\text{ and }
\sum_{i=1}^k f_{C_i} =f(G),
\end{equation}
and these two identities imply the following lemma.
\begin{lemma}\label{tri-b}
If $g(e_{C_i},f_{C_i})\leq 0$ for all $1\leq i\leq k$, then $ex_{\mathcal{P}}(n,H)\leq \frac{\alpha}{\beta}(n-2)$.
In particular, if the equality holds for some $n$, then every $g(e_{C_i},f_{C_i})= 0$ in the extremal construction.
\end{lemma}
For sake of simplicity, we  write $g(C_i)$ instead of $g(e_{C_i},f_{C_i})$.

\section{Proof of Theorem \ref{theta6-1}}

\subsection{Proof of the Lower Bound}
Here, we present a class of extremal constructions that satisfies the bound given by Theorem \ref{theta6-1}. 
We start with a $C_4$-free plane graph as a skeleton, in which every edge is adjacent to a $3$-face and a $5$-face. 
According to \cite{Dowden}, we can create infinitely many such graphs by alternatively gluing graph (a) and graph (b) in Figure \ref{fig:sharp}, where red pentagons and blue pentagons are matched up.
Then we replace every $3$-face (gray triangles) by $B_{5,a}$.
The resultant graphs contain $170k+70$ vertices and $450k+180$ edges, achieving the equality of $ex_{\mathcal{P}}(n,\Theta_6^1)\leq\frac{45}{17}(n-2)$.
% However, here, all of our triangular blocks are $B_{5,a}$ and not just a triangle.
% In figure \ref{fig:sharp}, every red-colored triangle is actually a $B_{5,a}$. We start with the graph labeled (a) in Figure \ref{fig:sharp}. Let (a) be graph $G_0$. We see that (a) has 70 vertices and 180 edges, so $\Theta(1) = -\frac{90}{17}$ in our example.
\begin{figure}[H] 
    \centering
\begin{center}
\begin{tabular}{c c}
\begin{tikzpicture}
    \filldraw[blue!30] ({0.5*sin(72)},{0.5*cos(72)})--({0.5*sin(144)},{0.5*cos(144)})--({0.5*sin(-144)},{0.5*cos(-144)})--({0.5*sin(-72)},{0.5*cos(-72)})--(0,0.5);
    \foreach \a in {0,...,4}
        \filldraw[gray!30] ({0.5*sin(72*\a)},{0.5*cos(72*\a)}) -- ({-sin(72*\a-144)},{-cos(72*\a-144)}) -- ({0.5*sin(72*\a+72)},{0.5*cos(72*\a+72)});
    \foreach \a in {0,...,4}
        \filldraw[gray!30] ({1.5*sin(72*\a+126)},{1.5*cos(72*\a+126)}) -- ({1.5*sin(72*\a+90)},{1.5*cos(72*\a+90)}) -- ({-sin(72*\a-72)},{-cos(72*\a-72)});
    \foreach \a in {0,...,4}
        \filldraw[gray!30] ({1.5*sin(72*\a+54)},{1.5*cos(72*\a+54)}) -- ({1.5*sin(72*\a+90)},{1.5*cos(72*\a+90)}) -- ({2*sin(72*\a+72)},{2*cos(72*\a+72)});
    \foreach \a in {0,...,4}
        \filldraw[gray!30] ({-3*sin(72*\a-72)},{-3*cos(72*\a-72)}) -- ({-3*sin(72*\a)},{-3*cos(72*\a)}) -- ({2*sin(72*\a+144)},{2*cos(72*\a+144)});
        
    \foreach \a in {0,...,4}
        \draw[very thick, blue] ({0.5*sin(72*\a)},{0.5*cos(72*\a)}) -- ({0.5*sin(72*\a+72)},{0.5*cos(72*\a+72)});
    \foreach \a in {0,...,4}
        \draw[thick] ({0.5*sin(72*\a)},{0.5*cos(72*\a)}) -- ({-sin(72*\a+144)},{-cos(72*\a+144)});
    \foreach \a in {0,...,4}
        \draw[thick] ({0.5*sin(72*\a)},{0.5*cos(72*\a)}) -- ({-sin(72*\a-144)},{-cos(72*\a-144)});
    \foreach \a in {0,...,9}
        \draw[thick] ({1.5*sin(36*\a+90)},{1.5*cos(36*\a+90)}) -- ({1.5*sin(36*\a+126)},{1.5*cos(36*\a+126)});
    \foreach \a in {0,...,4}
        \draw[thick] ({1.5*sin(72*\a+90)},{1.5*cos(72*\a+90)}) -- ({-sin(72*\a-72)},{-cos(72*\a-72)});
    \foreach \a in {0,...,4}
        \draw[thick] ({1.5*sin(72*\a+126)},{1.5*cos(72*\a+126)}) -- ({-sin(72*\a-72)},{-cos(72*\a-72)});
    \foreach \a in {0,...,4}
        \draw[thick] ({1.5*sin(72*\a+90)},{1.5*cos(72*\a+90)}) -- ({2*sin(72*\a+72)},{2*cos(72*\a+72)});
    \foreach \a in {0,...,4}
        \draw[thick] ({1.5*sin(72*\a+54)},{1.5*cos(72*\a+54)}) -- ({2*sin(72*\a+72)},{2*cos(72*\a+72)});
    \foreach \a in {0,...,4}
        \draw[very thick, red] ({-3*sin(72*\a)},{-3*cos(72*\a)}) -- ({-3*sin(72*\a+72)},{-3*cos(72*\a+72)});
    \foreach \a in {0,...,4}
        \draw[thick] ({-3*sin(72*\a)},{-3*cos(72*\a)}) -- ({2*sin(72*\a+144)},{2*cos(72*\a+144)});
    \foreach \a in {0,...,4}
        \draw[thick] ({-3*sin(72*\a)},{-3*cos(72*\a)}) -- ({2*sin(72*\a-144)},{2*cos(72*\a-144)});
        
    \foreach \a in {0,...,4}
        \filldraw ({0.5*sin(72*\a)},{0.5*cos(72*\a)})  circle (0.07);
    \foreach \a in {0,...,4}
        \filldraw ({-sin(72*\a)},{-cos(72*\a)})  circle (0.07);
    \foreach \a in {0,...,9}
        \filldraw ({1.5*sin(36*\a+90)},{1.5*cos(36*\a+90)})  circle (0.07);
    \foreach \a in {0,...,4}
        \filldraw ({2*sin(72*\a)},{2*cos(72*\a)})  circle (0.07);
    \foreach \a in {0,...,4}
        \filldraw ({-3*sin(72*\a)},{-3*cos(72*\a)})  circle (0.07);        
\end{tikzpicture}
&  
\begin{tikzpicture}
    \filldraw[red!30] ({-0.5*sin(72)},{-0.5*cos(72)})--({-0.5*sin(144)},{-0.5*cos(144)})--({-0.5*sin(-144)},{-0.5*cos(-144)})--({-0.5*sin(-72)},{-0.5*cos(-72)})--(0,-0.5);
    \foreach \a in {0,...,4}
        \filldraw[gray!30] ({0.8*sin(72*\a+90)},{0.8*cos(72*\a+90)}) -- ({0.8*sin(72*\a+126)},{0.8*cos(72*\a+126)}) -- ({-0.5*sin(72*\a-72)},{-0.5*cos(72*\a-72)});
    \foreach \a in {0,...,4}
        \filldraw[gray!30] ({1.2*sin(72*\a+72)},{1.2*cos(72*\a+72)}) -- ({0.8*sin(72*\a+90)},{0.8*cos(72*\a+90)}) -- ({1.6*sin(72*\a+90)},{1.6*cos(72*\a+90)});
    \foreach \a in {0,...,4}
        \filldraw[gray!30] ({1.2*sin(72*\a+72)},{1.2*cos(72*\a+72)}) -- ({0.8*sin(72*\a+54)},{0.8*cos(72*\a+54)}) -- ({1.6*sin(72*\a+54)},{1.6*cos(72*\a+54)});
    \foreach \a in {0,...,4}
        \filldraw[gray!30] ({-2.2*sin(72*\a)},{-2.2*cos(72*\a)}) -- ({-2.2*sin(72*\a+24)},{-2.2*cos(72*\a+24)}) -- ({1.6*sin(72*\a-162)},{1.6*cos(72*\a-162)});
    \foreach \a in {0,...,4}
        \filldraw[gray!30] ({-2.2*sin(72*\a)},{-2.2*cos(72*\a)}) -- ({-2.2*sin(72*\a-24)},{-2.2*cos(72*\a-24)}) -- ({1.6*sin(72*\a+162)},{1.6*cos(72*\a+162)});
    \foreach \a in {0,...,4}
        \filldraw[gray!30] ({-2.2*sin(72*\a-168)},{-2.2*cos(72*\a-168)}) -- ({3*sin(72*\a)},{3*cos(72*\a)}) -- ({-2.2*sin(72*\a+168)},{-2.2*cos(72*\a+168)});
        
    \foreach \a in {0,...,4}
        \draw[very thick, red] ({-0.5*sin(72*\a)},{-0.5*cos(72*\a)}) -- ({-0.5*sin(72*\a+72)},{-0.5*cos(72*\a+72)});
    \foreach \a in {0,...,4}
        \draw[thick] ({0.8*sin(72*\a+90)},{0.8*cos(72*\a+90)}) -- ({0.8*sin(72*\a+126)},{0.8*cos(72*\a+126)});
    \foreach \a in {0,...,4}
        \draw[thick] ({0.8*sin(72*\a+90)},{0.8*cos(72*\a+90)})  circle (0.07) -- ({-0.5*sin(72*\a-72)},{-0.5*cos(72*\a-72)})  circle (0.07);
    \foreach \a in {0,...,4}
        \draw[thick] ({0.8*sin(72*\a+126)},{0.8*cos(72*\a+126)})  circle (0.07) -- ({-0.5*sin(72*\a-72)},{-0.5*cos(72*\a-72)})  circle (0.07);
    \foreach \a in {0,...,4}
        \draw[thick] ({1.2*sin(72*\a+72)},{1.2*cos(72*\a+72)}) -- ({0.8*sin(72*\a+90)},{0.8*cos(72*\a+90)});
    \foreach \a in {0,...,4}
        \draw[thick] ({1.2*sin(72*\a+72)},{1.2*cos(72*\a+72)}) -- ({0.8*sin(72*\a+54)},{0.8*cos(72*\a+54)});
    \foreach \a in {0,...,9}
        \draw[thick] ({0.8*sin(36*\a+90)},{0.8*cos(36*\a+90)}) -- ({1.6*sin(36*\a+90)},{1.6*cos(36*\a+90)});
    \foreach \a in {0,...,4}
        \draw[thick] ({1.2*sin(72*\a+72)},{1.2*cos(72*\a+72)}) -- ({1.6*sin(72*\a+90)},{1.6*cos(72*\a+90)});
    \foreach \a in {0,...,4}
        \draw[thick] ({1.2*sin(72*\a+72)},{1.2*cos(72*\a+72)}) -- ({1.6*sin(72*\a+54)},{1.6*cos(72*\a+54)});
    \foreach \a in {0,...,4}
        \draw[thick] ({-2.2*sin(72*\a-120)},{-2.2*cos(72*\a-120)}) -- ({1.6*sin(72*\a+54)},{1.6*cos(72*\a+54)});
    \foreach \a in {0,...,4}
        \draw[thick] ({-2.2*sin(72*\a-96)},{-2.2*cos(72*\a-96)}) -- ({1.6*sin(72*\a+90)},{1.6*cos(72*\a+90)});
    \foreach \a in {0,...,4}
        \draw[thick] ({-2.2*sin(72*\a-72)},{-2.2*cos(72*\a-72)}) -- ({1.6*sin(72*\a+90)},{1.6*cos(72*\a+90)});
    \foreach \a in {0,...,4}
        \draw[thick] ({-2.2*sin(72*\a-72)},{-2.2*cos(72*\a-72)}) -- ({1.6*sin(72*\a+126)},{1.6*cos(72*\a+126)});
    \foreach \a in {0,...,14}
        \draw[thick] ({-2.2*sin(24*\a)},{-2.2*cos(24*\a)}) -- ({-2.2*sin(24*\a+24)},{-2.2*cos(24*\a+24)});
    \foreach \a in {0,...,4}
        \draw[thick] ({-2.2*sin(72*\a+168)},{-2.2*cos(72*\a+168)}) -- ({3*sin(72*\a)},{3*cos(72*\a)});
    \foreach \a in {0,...,4}
        \draw[thick] ({-2.2*sin(72*\a-168)},{-2.2*cos(72*\a-168)}) -- ({3*sin(72*\a)},{3*cos(72*\a)});
    \foreach \a in {0,...,4}
        \draw[very thick, blue] ({3*sin(72*\a)},{3*cos(72*\a)}) -- ({3*sin(72*\a+72)},{3*cos(72*\a+72)});
        
    \foreach \a in {0,...,4}
        \filldraw ({-0.5*sin(72*\a)},{-0.5*cos(72*\a)})  circle (0.07);
    \foreach \a in {0,...,9}
        \filldraw ({0.8*sin(36*\a+90)},{0.8*cos(36*\a+90)})  circle (0.07);
    \foreach \a in {0,...,4}
        \filldraw ({1.2*sin(72*\a)},{1.2*cos(72*\a)})  circle (0.07);
    \foreach \a in {0,...,9}
        \filldraw ({1.6*sin(36*\a+90)},{1.6*cos(36*\a+90)})  circle (0.07);
    \foreach \a in {0,...,14}
        \filldraw ({-2.2*sin(24*\a)},{-2.2*cos(24*\a)})  circle (0.07);        
    \foreach \a in {0,...,4}
        \filldraw ({3*sin(72*\a)},{3*cos(72*\a)})  circle (0.07);
          
    % \foreach \a in {0,...,4}
    %     \draw[thick]
\end{tikzpicture}
\\

(a) & (b)
\end{tabular}
\end{center}
\caption{Extremal Constructions of $\Theta_6^1$-free Graph}
    \label{fig:sharp}
\end{figure}
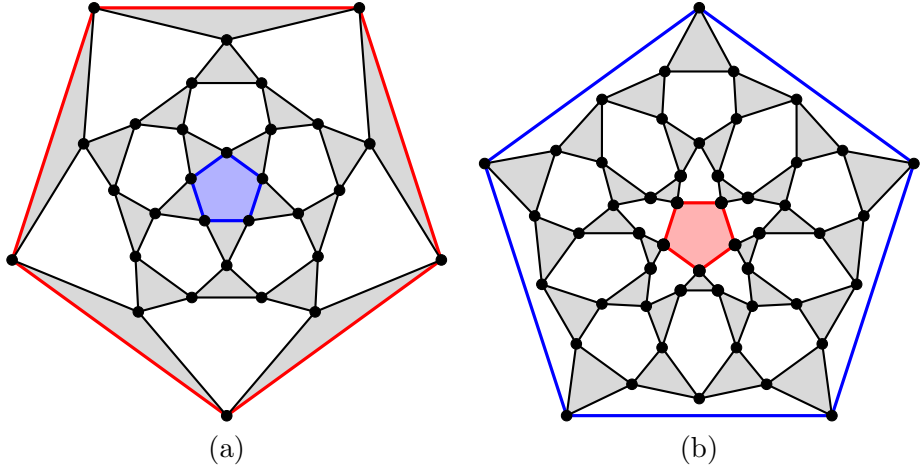
% Building from (a), we place (b) around (a) so that the red-colored edges match up, and then place (a) around (b) so  that the blue-colored edges match up to get a larger construction. This iteration adds $105 + 65 = 170$ vertices and $275 + 175 = 450$ edges to (a).\\ \\
% Thus, for each $k^{th}$ iteration, $n = 170k + 70$ and $e(G_k) = 450k + 180$. Then, we see that $e(G_k) = \frac{45}{17}(n-2)$ for each iteration.
\subsection{Proof of the Upper Bound}
Let $G$ be a $\Theta_6^1$-free plane graph and $\mathcal{B}$ be the set of its (trivial) triangular blocks.
We will partition $\mathcal{B}$ into clusters, but we postpone explaining how exactly we divide it.
In order to apply Lemma \ref{tri-b}, we convert the target inequality $ex_{\mathcal{P}}(n,\Theta_6^1)\leq\frac{45}{17}(n-2)$ to the contribution formula
\begin{align*}
    g(C)=g(e_{C},f_{C})=45f_{C}-28e_{C}
\end{align*}
for each cluster $C$.
% that this construction and corresponding bound is the best construction for a $\Theta_6^1$-free graph by calculating the contributions of each cluster. 
% From our extremal construction, we obtain the target inequality $ex_{\mathcal{P}}(n,\Theta_6^1)\leq \frac{45}{17}(n-2)$. Thus, our contribution formula for each cluster $C_i$ is
Inspired by the extremal constructions, we want to show that the contribution $g$ is zero for $B_{5,a}$ and is strictly less than zero for other clusters. 

Throughout this section, we assume that $G$ has at least $6$ vertices, i.e., $n \geq 6$. 
This condition implies that for each triangular block $B$ of $G$ with $n(B)\leq 5$, one of the faces of $B$ is not an interior face, otherwise $B$ has to be entire $G$.
We thus always draw the non-interior face as the unbounded face.
Also, we assume that $G$ is connected, otherwise we can freely add an edge between two connected components of $G$ without creating a $\Theta^1_6$.

By Lemma \ref{poss_tri_block}, there are nine possible blocks in $\mathcal{B}$.
Let us estimate the contribution of them one-by-one.

\subsubsection{$g(B_{2})\leq -\frac{11}{2}.$}

\begin{center}
\begin{tikzpicture}
\filldraw (-1,0) circle (0.07);
\filldraw (1,0) circle (0.07);
\draw[thick] (-1,0) -- (1,0);
\node[left] at (-1,0) {$x_1$};
\node[right] at (1,0) {$x_2$};
\end{tikzpicture}
\end{center}
\textit{Proof.}
Notice that $B_2$ has no interior faces.
Pick $f\in F(G)$ adjacent to $x_1x_2$.
If $f$ is a face of length $2$, then $B_2$ is the entire $G$, which is impossible as $G$ has at least $6$ vertices.
Also, $f$ cannot be a face of length $3$ by Remark \ref{adj}. 
Thus, $f$ has length at least $4$, and the this argument holds for each side of $x_1x_2$.
Therefore, $f_{B_2}\leq \frac{1}{4}\cdot 2$, $e_{B_2}=1$, and $g(B_2)\leq -\frac{11}{2}$.
%Let $B_2$ consist of two vertices $x_1$ and $x_2$ and an edge $x_1x_2$. 
%because that would contradict our definition of a triangular block.
%face joining each side of $x_1x_2$ must be an exterior face and have length at least $4$. Therefore, we have $\hat{f}(B) \leq \frac{1}{4} + \frac{1}{4} = \frac{1}{2}$. Using $e(B) = 1$, we obtain $45\hat{f}(B) - 28e(B) \leq 45 \cdot \frac{1}{2} - 28 \cdot 1 = -\frac{11}{2} < 0$.

\subsubsection{$g(B_3)\leq -\frac{21}{4}.$}
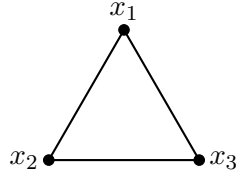
\begin{center}
\begin{tikzpicture}
\filldraw (-1,0) circle (0.07);
\filldraw (1,0) circle (0.07);
\filldraw (0,1.732) circle (0.07);
\draw[thick] (-1,0) -- (1,0);
\draw[thick] (-1,0) -- (0,1.732);
\draw[thick] (0,1.732) -- (1,0);
\node[above] at (0,1.732) {$x_1$};
\node[left] at (-1,0) {$x_2$};
\node[right] at (1,0) {$x_3$};
\end{tikzpicture}
\end{center}
\textit{Proof.}
%Let $B$ consist of three vertices $x_1$, $x_2$, and $x_3$ and three edges $x_1x_2$, $x_2x_3$, and $x_1x_3$. 
%Because $G$ has more than $6$ vertices, the face joining each of $x_1x_2$, $x_2x_3$ and $x_3x_1$ must be either $x_1x_2x_3$ or an exterior face.
The only interior face of $B_3$ is the bounded $3$-face $x_1x_2x_3$.
For $f\in F(G)$ adjacent to $x_1x_2$, $x_2x_3$, or $x_1x_3$ on the outside, we know it has length at least $4$ by Remark \ref{adj}.
Therefore, $f_{B_3}\leq 1+ \frac{1}{4}\cdot 3$, $e_{B_3}=3$, and $g(B_3)\leq -\frac{21}{4}$.

\subsubsection{$g(B_{4,a})\leq -\frac{3}{2}.$}
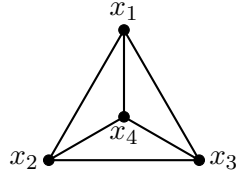
\begin{center}
\begin{tikzpicture}
\filldraw (-1,0) circle (0.07);
\filldraw (1,0) circle (0.07);
\filldraw (0,1.732) circle (0.07);
\filldraw (0,0.577) circle (0.07);
\draw[thick] (-1,0) -- (1,0);
\draw[thick] (-1,0) -- (0,1.732);
\draw[thick] (0,1.732) -- (1,0);
\draw[thick] (-1,0) -- (0,0.577);
\draw[thick] (1,0) -- (0,0.577);
\draw[thick] (0,1.732) -- (0,0.577);
\node[above] at (0,1.732) {$x_1$};
\node[left] at (-1,0) {$x_2$};
\node[right] at (1,0) {$x_3$};
\node[below] at (0,0.577) {$x_4$};
\end{tikzpicture}
\end{center}
\textit{Proof.} 
The interior faces of $B_{4,a}$ are the three bounded $3$-faces in the figure above. 
We will show that the faces of $G$ adjacent to $x_1x_2$, $x_1x_3$, and $x_2x_3$ on the outside cannot all have length $4$.
Suppose $f\in F(G)$ is the $4$-face adjacent to $x_1x_2$ on the outside.
Then there are two possibilities: $f$ shares one common edge with $B_{4,a}$ or two common edges with $B_{4,a}$.

If $f$ shares one common edge with $B_{4,a}$, then there exist two additional vertices $y_1$ and $y_2$ such that $x_1x_2y_1y_2$ is a $4$-face.
However, there is a contradiction that such a configuration contains $x_1x_2y_1y_2\cup x_1x_2x_4x_3$ isomorphic to $\Theta_6^1$.
    % \label{theta61K4case1} The bounded $4$-face $x_1x_2y_1y_2$ joins $B$ in exactly one edge $x_1x_2$. Clearly, neither $y_1$ nor $y_2$ can be $x_3$.
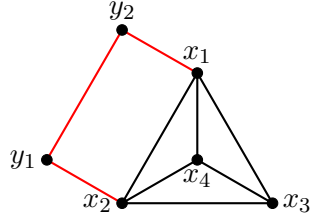
\begin{center}
\begin{tikzpicture}
\draw[thick] (-1,0) -- (1,0);
\draw[thick] (-1,0) -- (0,1.732);
\draw[thick] (0,1.732) -- (1,0);
\draw[thick] (-1,0) -- (0,0.577);
\draw[thick] (1,0) -- (0,0.577);
\draw[thick] (0,1.732) -- (0,0.577);
\draw[thick,red] (-2,0.577) -- (-1,0);
\draw[thick,red] (-2,0.577) -- (-1,2.305);
\draw[thick,red] (0,1.732) -- (-1,2.305);
\node[above] at (0,1.732) {$x_1$};
\node[left] at (-1,0) {$x_2$};
\node[right] at (1,0) {$x_3$};
\node[below] at (0,0.577) {$x_4$};
\node[left] at (-2, 0.577) {$y_1$};
\node[above] at (-1,2.305) {$y_2$};
\filldraw (-1,0) circle (0.07);
\filldraw (1,0) circle (0.07);
\filldraw (0,1.732) circle (0.07);
\filldraw (0,0.577) circle (0.07);
\filldraw (-2,0.577) circle (0.07);
\filldraw (-1,2.305) circle (0.07);
\end{tikzpicture}
\end{center}
If $f$ shares two common edges with $B_{4,a}$, then there exists an additional vertex $y$ such that $x_1x_2yx_3$ is a $4$-face.
By the previous paragraph, we know that $f'\in F(G)$ adjacent to $x_2x_3$ on the outside cannot be a $4$-face as it cannot share two common edges with $B_{4,a}$.
% The unbounded $4$-face $x_1x_2y_1y_2$ joins $B$ in exactly one edge $x_1x_2$. Clearly, neither $y_1$ nor $y_2$ can be $x_3$.
    % \begin{center}
    % \begin{tikzpicture}
    % \filldraw (-1,0) circle (0.07);
    % \filldraw (1,0) circle (0.07);
    % \filldraw (0,1.732) circle (0.07);
    % \filldraw (0,0.577) circle (0.07);
    % \filldraw (1.5,-1.442) circle (0.07);
    % \filldraw (2.5,0.29) circle (0.07);
    % \draw[thick] (-1,0) -- (1,0);
    % \draw[thick] (-1,0) -- (0,1.732);
    % \draw[thick] (0,1.732) -- (1,0);
    % \draw[thick] (-1,0) -- (0,0.577);
    % \draw[thick] (1,0) -- (0,0.577);
    % \draw[thick] (0,1.732) -- (0,0.577);
    % \draw[thick,red] (1.5,-1.442) -- (-1,0);
    % \draw[thick,red] (2.5,0.29) -- (0,1.732);
    % \draw[thick,red] (1.5,-1.442) -- (2.5,0.29);
    % \node[above] at (0,1.732) {$x_1$};
    % \node[left] at (-1,0) {$x_2$};
    % \node[right] at (1,0) {$x_3$};
    % \node[below] at (0,0.577) {$x_4$};
    % \node[left] at (1.5,-1.442) {$y_1$};
    % \node[above] at (2.5,0.29) {$y_2$};
    % \end{tikzpicture}
    % \end{center}
    % Then the path $x_1x_2y_1y_2x_1x_3x_4x_2$ also forms a $\Theta^1_6$, a contradiction.

    % \label{theta61K4case3} The bounded $4$-face joins another exterior edge of $B$ besides $x_1x_2$, WLOG say $x_1x_2yx_3$ joins $x_1x_2$ and $x_1x_3$;
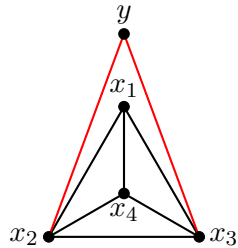
\begin{center}
\begin{tikzpicture}
\draw[thick] (-1,0) -- (1,0);
\draw[thick] (-1,0) -- (0,1.732);
\draw[thick] (0,1.732) -- (1,0);
\draw[thick] (-1,0) -- (0,0.577);
\draw[thick] (1,0) -- (0,0.577);
\draw[thick] (0,1.732) -- (0,0.577);
\draw[thick,red] (-1,0) -- (0,2.7);
\draw[thick,red] (1,0) -- (0,2.7);
\node[above] at (0,1.732) {$x_1$};
\node[left] at (-1,0) {$x_2$};
\node[right] at (1,0) {$x_3$};
\node[below] at (0,0.577) {$x_4$};
\node[above] at (0,2.7) {$y$};
\filldraw (-1,0) circle (0.07);
\filldraw (1,0) circle (0.07);
\filldraw (0,1.732) circle (0.07);
\filldraw (0,0.577) circle (0.07);
\filldraw (0,2.7) circle (0.07);
\end{tikzpicture}
\end{center} 
We have thus shown that one of the faces adjacent to $x_1x_2$, $x_1x_3$, and $x_2x_3$ has length at least $5$.
Therefore, $f_{B_{4,a}}\leq 3+\frac{1}{4}\cdot 2+\frac{1}{5}$, $e_{B_{4,a}}=6$, and $g(B_{4,a})\leq -\frac{3}{2}$.

\subsubsection{$g(B_{4,b})\leq -5.$}
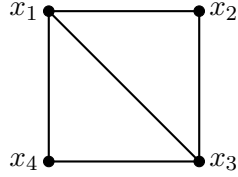
\begin{center}
\begin{tikzpicture}
\filldraw (-1,1) circle (0.07);
\filldraw (1,1) circle (0.07);
\filldraw (1,-1) circle (0.07);
\filldraw (-1,-1) circle (0.07);
\draw[thick] (-1,1) -- (1,1);
\draw[thick] (-1,1) -- (-1,-1);
\draw[thick] (1,-1) -- (1,1);
\draw[thick] (1,-1) -- (-1,-1);
\draw[thick] (-1,1) -- (1,-1);
\node[left] at (-1,1) {$x_1$};
\node[right] at (1,1) {$x_2$};
\node[left] at (-1,-1) {$x_4$};
\node[right] at (1,-1) {$x_3$};
\end{tikzpicture}
\end{center}
\textit{Proof.} 
The interior faces of $B_{4,b}$ are the two bounded $3$-faces in the figure above. 
For $f\in F(G)$ adjacent to $x_1x_2$, $x_2x_3$, $x_3x_4$, or $x_1x_4$ on the outside, we know it has length at least $4$ by Remark \ref{adj}.
Therefore, $f_{B_{4,b}}\leq 2+\frac{1}{4}\cdot 4$, $e_{B_{4,b}}=5$, and $g(B_{4,b})\leq -5$.

% We only need to calculate the contribution of the case with the largest face contribution, since all other cases will have a smaller contribution. This case has two adjacent faces with 4 sides, where $x_2x_1x_4$ and $x_2x_3x_4$ are part of separate exterior faces. Then, $f_{B_{4,b}}\leq 2+4/4$ and $g(B_{4,b})\leq -5$.
\subsubsection{$g(B_{5,a}) \leq 0.$}
\begin{center}
\begin{tikzpicture}
\draw[fill=gray!42] (1.5,0) -- (0,1.7) -- (0,0.87);
\filldraw (-1.5,0) circle (0.07);
\filldraw (1.5,0) circle (0.07);
\filldraw (0,2.6) circle (0.07);
\filldraw (0,0.87) circle (0.07);
\filldraw (0,1.7) circle (0.07);
\draw[thick] (-1.5,0) -- (1.5,0);
\draw[thick] (-1.5,0) -- (0,2.6);
\draw[thick] (0,2.6) -- (1.5,0);
\draw[thick] (-1.5,0) -- (0,0.87);
\draw[thick] (1.5,0) -- (0,0.87);
\draw[thick] (0,2.6) -- (0,0.87);
\draw[thick] (-1.5,0) -- (0,1.7);
\draw[thick] (1.5,0) -- (0,1.7);
\draw[thick] (0,2.6) -- (0,1.7);
\node[above] at (0,2.6) {$x_1$};
\node[left] at (-1.5,0) {$x_2$};
\node[right] at (1.5,0) {$x_3$};
\node[right] at (0,1.64) {$x_4$};
\node[below] at (0,0.87) {$x_5$};
\end{tikzpicture}
\end{center}
\textit{Proof.} 
For this block, the set of interior faces has two possibilities: 
it contains all five bounded $3$-faces or the four unshaded, bounded $3$-faces as shown in the figure above.

Suppose the interior faces are all five bounded $3$-faces. 
Pick $f\in F(G)$ adjacent to $B_{5,a}$ on the outside. 
By Remark \ref{adj}, $f$ has length at least $4$.
Similar to the case of $B_{4,a}$, there are four ways to have an adjacent $4$-face as shown in the figure below, and all of them are not $\Theta_6^1$-free.
Hence, $f$ has length at least $5$, and this argument holds for $x_1x_2$, $x_1x_3$, and $x_2x_3$ simultaneously.
Therefore, in this case, the contribution to the face number is at most $5+\frac{1}{5}\cdot 3=\frac{28}{5}$.

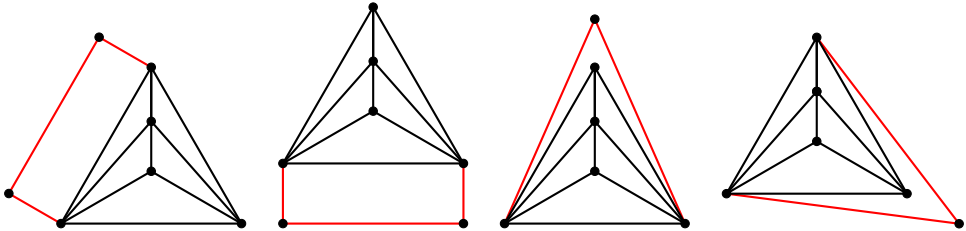
\begin{center}
\begin{tabular}{c c c c}
\begin{tikzpicture}[scale=0.8]
\draw[thick,red] (-1.5,0) -- (-2.366,0.5);
\draw[thick,red] (-0.866,3.1) -- (-2.366,0.5);
\draw[thick,red] (-0.866,3.1) -- (0,2.6);
\draw[thick] (-1.5,0) -- (1.5,0);
\draw[thick] (-1.5,0) -- (0,2.6);
\draw[thick] (0,2.6) -- (1.5,0);
\draw[thick] (-1.5,0) -- (0,0.87);
\draw[thick] (1.5,0) -- (0,0.87);
\draw[thick] (0,2.6) -- (0,0.87);
\draw[thick] (-1.5,0) -- (0,1.7);
\draw[thick] (1.5,0) -- (0,1.7);
\draw[thick] (0,2.6) -- (0,1.7);
\filldraw (-1.5,0) circle (0.07);
\filldraw (1.5,0) circle (0.07);
\filldraw (0,2.6) circle (0.07);
\filldraw (0,0.87) circle (0.07);
\filldraw (0,1.7) circle (0.07);
\filldraw (-2.366,0.5) circle (0.07);
\filldraw (-0.866,3.1) circle (0.07);
\end{tikzpicture}
&
\begin{tikzpicture}[scale=0.8]
\draw[thick,red] (-1.5,0) -- (-1.5,-1);
\draw[thick,red] (1.5,-1) -- (-1.5,-1);
\draw[thick,red] (1.5,-1) -- (1.5,0);
\draw[thick] (-1.5,0) -- (1.5,0);
\draw[thick] (-1.5,0) -- (0,2.6);
\draw[thick] (0,2.6) -- (1.5,0);
\draw[thick] (-1.5,0) -- (0,0.87);
\draw[thick] (1.5,0) -- (0,0.87);
\draw[thick] (0,2.6) -- (0,0.87);
\draw[thick] (-1.5,0) -- (0,1.7);
\draw[thick] (1.5,0) -- (0,1.7);
\draw[thick] (0,2.6) -- (0,1.7);
\filldraw (-1.5,0) circle (0.07);
\filldraw (1.5,0) circle (0.07);
\filldraw (0,2.6) circle (0.07);
\filldraw (0,0.87) circle (0.07);
\filldraw (0,1.7) circle (0.07);
\filldraw (-1.5,-1) circle (0.07);
\filldraw (1.5,-1) circle (0.07);
\end{tikzpicture}
&
\begin{tikzpicture}[scale=0.8]
\draw[thick,red] (0,3.4) -- (-1.5,0);
\draw[thick,red] (0,3.4) -- (1.5,0);
\draw[thick] (-1.5,0) -- (1.5,0);
\draw[thick] (-1.5,0) -- (0,2.6);
\draw[thick] (0,2.6) -- (1.5,0);
\draw[thick] (-1.5,0) -- (0,0.87);
\draw[thick] (1.5,0) -- (0,0.87);
\draw[thick] (0,2.6) -- (0,0.87);
\draw[thick] (-1.5,0) -- (0,1.7);
\draw[thick] (1.5,0) -- (0,1.7);
\draw[thick] (0,2.6) -- (0,1.7);
\filldraw (-1.5,0) circle (0.07);
\filldraw (1.5,0) circle (0.07);
\filldraw (0,2.6) circle (0.07);
\filldraw (0,0.87) circle (0.07);
\filldraw (0,1.7) circle (0.07);
\filldraw (0,3.4) circle (0.07);
\end{tikzpicture}
&
\begin{tikzpicture}[scale=0.8]
\draw[thick,red] (2.366,-0.5) -- (-1.5,0);
\draw[thick,red] (2.366,-0.5) -- (0,2.6);
\draw[thick] (-1.5,0) -- (1.5,0);
\draw[thick] (-1.5,0) -- (0,2.6);
\draw[thick] (0,2.6) -- (1.5,0);
\draw[thick] (-1.5,0) -- (0,0.87);
\draw[thick] (1.5,0) -- (0,0.87);
\draw[thick] (0,2.6) -- (0,0.87);
\draw[thick] (-1.5,0) -- (0,1.7);
\draw[thick] (1.5,0) -- (0,1.7);
\draw[thick] (0,2.6) -- (0,1.7);
\filldraw (-1.5,0) circle (0.07);
\filldraw (1.5,0) circle (0.07);
\filldraw (0,2.6) circle (0.07);
\filldraw (0,0.87) circle (0.07);
\filldraw (0,1.7) circle (0.07);
\filldraw (0,1.7) circle (0.07);
\filldraw (2.366,-0.5) circle (0.07);
\end{tikzpicture}
\end{tabular}
\end{center}

Suppose the interior faces are the four unshaded, bounded $3$-faces.
For any $f\in F(G)$ adjacent to $B_{5,a}$ on the outside, the argument in the previous paragraph still works, so such $f$ has length at least $5$.
For $f'\in F(G)$ adjacent to $x_3x_4$, $x_3x_5$, or $x_4x_5$ in the shaded region, we know it has length at least $4$ by Remark \ref{adj}.
Therefore, in this case, the contribution to the face number is at most $4+\frac{1}{5}\cdot 3+\frac{1}{4}\cdot 3=\frac{107}{20}$.

Since the first case has strictly larger contribution to the face number, we have $f_{B_{5,a}}\leq\frac{28}{5}$, $e_{B_{5,a}}=9$, and $g(B_{5,a})\leq 0$.
Notice that the equality is achieved when all bounded $3$-faces are interior faces, and when every face adjacent to $B_{5,a}$ on the outside is a $5$-face.

\subsubsection{$g(B_{5,b})\leq -8.$}
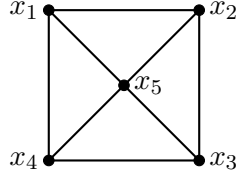
\begin{center}
\begin{tikzpicture}
\filldraw (-1,1) circle (0.07);
\filldraw (1,1) circle (0.07);
\filldraw (1,-1) circle (0.07);
\filldraw (-1,-1) circle (0.07);
\filldraw (0,0) circle (0.07);
\draw[thick] (-1,1) -- (1,1);
\draw[thick] (-1,1) -- (-1,-1);
\draw[thick] (1,-1) -- (1,1);
\draw[thick] (1,-1) -- (-1,-1);
\draw[thick] (1,1) -- (0,0);
\draw[thick] (-1,-1) -- (0,0);
\draw[thick] (-1,1) -- (0,0);
\draw[thick] (1,-1) -- (0,0);
\node[left] at (-1,1) {$x_1$};
\node[right] at (1,1) {$x_2$};
\node[left] at (-1,-1) {$x_4$};
\node[right] at (1,-1) {$x_3$};
\node[right] at (0,0) {$x_5$};
\end{tikzpicture}
\end{center}
\textit{Proof.} 
The interior faces of $B_{5,b}$ are the four bounded $3$-faces in the figure above. 
Pick $f\in F(G)$ adjacent to $x_1x_2$ on the outside. 
By Remark \ref{adj}, $f$ has length at least $4$.
There are two ways to have an adjacent $4$-face as shown in the figure below, and both of them are not $\Theta_6^1$-free.
Hence, $f$ has length at least $5$, and this argument holds for $x_2x_3$, $x_3x_4$, and $x_1x_4$ as well.
Therefore, $f_{B_{5,b}}\leq 4+\frac{1}{5}\cdot 4$, $e_{B_{5,b}}=8$, and $g(B_{5,b})\leq -8$.
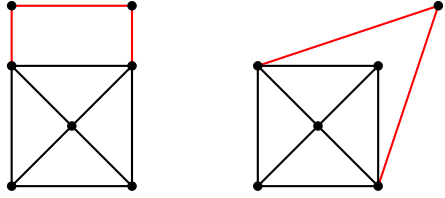
\begin{center}
\begin{tabular}{c c}
\begin{tikzpicture}[scale=0.8]
\draw[thick,red] (1,2) -- (1,1);
\draw[thick,red] (1,2) -- (-1,2);
\draw[thick,red] (-1,1) -- (-1,2);
\draw[thick] (-1,1) -- (1,1);
\draw[thick] (-1,1) -- (-1,-1);
\draw[thick] (1,-1) -- (1,1);
\draw[thick] (1,-1) -- (-1,-1);
\draw[thick] (1,1) -- (0,0);
\draw[thick] (-1,-1) -- (0,0);
\draw[thick] (-1,1) -- (0,0);
\draw[thick] (1,-1) -- (0,0);
\filldraw (-1,1) circle (0.07);
\filldraw (1,1) circle (0.07);
\filldraw (1,-1) circle (0.07);
\filldraw (-1,-1) circle (0.07);
\filldraw (0,0) circle (0.07);
\filldraw (1,2) circle (0.07);
\filldraw (-1,2) circle (0.07);
\end{tikzpicture}
&
\hspace{1cm}
\begin{tikzpicture}[scale=0.8]
\draw[thick,red] (2,2) -- (1,-1);
\draw[thick,red] (2,2) -- (-1,1);
\draw[thick] (-1,1) -- (1,1);
\draw[thick] (-1,1) -- (-1,-1);
\draw[thick] (1,-1) -- (1,1);
\draw[thick] (1,-1) -- (-1,-1);
\draw[thick] (1,1) -- (0,0);
\draw[thick] (-1,-1) -- (0,0);
\draw[thick] (-1,1) -- (0,0);
\draw[thick] (1,-1) -- (0,0);
\filldraw (-1,1) circle (0.07);
\filldraw (1,1) circle (0.07);
\filldraw (1,-1) circle (0.07);
\filldraw (-1,-1) circle (0.07);
\filldraw (0,0) circle (0.07);
\filldraw (2,2) circle (0.07);
\end{tikzpicture}
\end{tabular}
\end{center}

\subsubsection{$g(B_{5,d})\leq  -\frac{19}{4}.$}
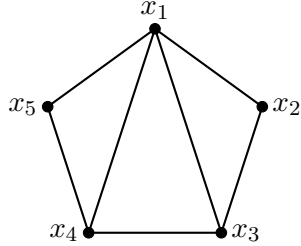
\begin{center}
\begin{tikzpicture}
\filldraw (0,1.5) circle (0.07);
\filldraw (0.951*1.5,0.309*1.5) circle (0.07);
\filldraw (-0.951*1.5,0.309*1.5) circle (0.07);
\filldraw (0.588*1.5,-0.809*1.5) circle (0.07);
\filldraw (-0.588*1.5,-0.809*1.5) circle (0.07);
\draw[thick] (0.951*1.5,0.309*1.5) -- (0,1.5);
\draw[thick] (0.951*1.5,0.309*1.5) -- (0.588*1.5,-0.809*1.5);
\draw[thick] (-0.588*1.5,-0.809*1.5) -- (0.588*1.5,-0.809*1.5);
\draw[thick] (-0.588*1.5,-0.809*1.5) -- (-0.951*1.5,0.309*1.5);
\draw[thick] (0,1.5) -- (-0.951*1.5,0.309*1.5);
\draw[thick] (0,1.5) -- (0.588*1.5,-0.809*1.5);
\draw[thick] (0,1.5) -- (-0.588*1.5,-0.809*1.5);
\node[above] at (0,1.5) {$x_1$};
\node[right] at (0.951*1.5,0.309*1.5) {$x_2$};
\node[right] at (0.588*1.5,-0.809*1.5) {$x_3$};
\node[left] at (-0.588*1.5,-0.809*1.5) {$x_4$};
\node[left] at (-0.951*1.5,0.309*1.5) {$x_5$};
\end{tikzpicture}
\end{center}

\textit{Proof.} 
The interior faces of $B_{5,d}$ are the three bounded $3$-faces in the figure above. 
For $f\in F(G)$ adjacent to $x_1x_2$, $x_2x_3$, $x_3x_4$, $x_4x_5$, or $x_1x_5$ on the outside, we know it has length at least $4$ by Remark \ref{adj}.
Therefore, $f_{B_{5,d}}\leq 3+\frac{1}{4}\cdot 5$, $e_{B_{5,d}}=7$, and $g(B_{5,d})\leq -\frac{19}{4}$.

\subsubsection{$g(B_6)\leq -\frac{9}{2}.$}
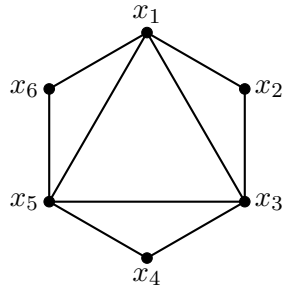
\begin{center}
\begin{tikzpicture}
\filldraw (0,1.5) circle (0.07);
\filldraw (0,-1.5) circle (0.07);
\filldraw (-1.3,0.75) circle (0.07);
\filldraw (1.3,0.75) circle (0.07);
\filldraw (-1.3,-0.75) circle (0.07);
\filldraw (1.3,-0.75) circle (0.07);
\draw[thick] (0,1.5) -- (-1.3,0.75);
\draw[thick] (-1.3,-0.75) -- (-1.3,0.75);
\draw[thick] (-1.3,-0.75) -- (0,-1.5);
\draw[thick] (0,1.5) -- (1.3,0.75);
\draw[thick] (1.3,-0.75) -- (1.3,0.75);
\draw[thick] (1.3,-0.75) -- (0,-1.5);
\draw[thick] (0,1.5) -- (-1.3,-0.75);
\draw[thick] (-1.3,-0.75) -- (1.3,-0.75);
\draw[thick] (1.3,-0.75) -- (0,1.5);
\node[above] at (0,1.5) {$x_1$};
\node[right] at (1.3,0.75) {$x_2$};
\node[right] at (1.3,-0.75) {$x_3$};
\node[below] at (0,-1.5) {$x_4$};
\node[left] at (-1.3,-0.75) {$x_5$};
\node[left] at (-1.3,0.75) {$x_6$};
\end{tikzpicture}
\end{center}
\textit{Proof.} 
The interior faces of $B_{6}$ are the four bounded $3$-faces in the figure above. 
For $f\in F(G)$ adjacent to $x_1x_2$, $x_2x_3$, $x_3x_4$, $x_4x_5$, $x_5x_6$ or $x_1x_6$ on the outside, we know it has length at least $4$ by Remark \ref{adj}.
Therefore, $f_{B_{6}}\leq 4+\frac{1}{4}\cdot 6$, $e_{B_{4,b}}=9$, and $g(B_{4,b})\leq -\frac{9}{2}$.

% Notice that the exterior faces joining each of the exterior edge $x_1x_2$, $x_2x_3$, $x_3x_4$, $x_4x_5$, $x_5x_6$ and $x_1x_6$ must be at least a $4$-face; otherwise, there is a new exterior $3$-face joining $B$, contradicting the fact that $B = B_6$ cannot be extended to a larger triangular block. Since $B$ has $4$ interior faces, so $\hat{f}(B) \leq 4 + 6 \cdot \frac{1}{4} = \frac{11}{2}$. Since $e(B) = 9$, so $45\hat{f}(B) - 28e(B) \leq 45 \cdot \frac{11}{2} - 28 \cdot 9 = -\frac{9}{2} < 0$.

\subsubsection{Troublemaker $B_{5,c}$}\label{troublemaker_Theta_61}
\begin{center}
\begin{tikzpicture}
\filldraw (0,1) circle (0.07);
\filldraw (0,-1) circle (0.07);
\filldraw (-1.732,0) circle (0.07);
\filldraw (1.732,0) circle (0.07);
\filldraw (0.577,0) circle (0.07);
\draw[thick] (0,1) -- (0,-1);
\draw[thick] (0,1) -- (-1.732,0);
\draw[thick] (0,1) -- (1.732,0);
\draw[thick] (0,-1) -- (-1.732,0);
\draw[thick] (0,-1) -- (1.732,0);
\draw[thick] (0,1) -- (0.577,0);
\draw[thick] (0,-1) -- (0.577,0);
\draw[thick] (1.732,0) -- (0.577,0);
\node[above] at (0,1) {$x_1$};
\node[right] at (1.732,0) {$x_2$};
\node[below] at (0,-1) {$x_3$};
\node[left] at (-1.732,0) {$x_4$};
\node[below] at (0.7,0) {$x_5$};
\end{tikzpicture}
\end{center}
The interior faces of $B_{5,c}$ are the four bounded $3$-faces in the figure above. 
By Remark \ref{adj}, every $f\in F(G)$ adjacent to $B_{5,c}$ on the outside has length at least $4$.
Unfortunately, this estimation is not enough to show $g(B_{5,c})
<0$, so we will use clusters to solve this issue.

% Let $B$ be a $B_{5,c}$ consisting of five vertices $x_1$, $x_2$, $x_3$, $x_4$ and $x_5$ and eight edges $x_1x_2$, $x_2x_3$, $x_3x_4$, $x_1x_4$, $x_1x_3$, $x_1x_5$, $x_2x_5$ and $x_3x_5$. Notice that the exterior face joining each of $x_1x_2$, $x_2x_3$, $x_3x_4$ and $x_1x_4$ must be at least a $4$-face (otherwise we can build a larger triangular block based on $B$, which is a contradiction). So each edge would contribute at most $\frac{1}{4}$ to the overall $\hat{f}(B)$. Since $B$ has four interior faces, so $\hat{f}(B) \leq 4 + 4 \cdot \frac{1}{4} = 5$. Since $e(B) = 8$, so $45\hat{f}(B) - 28e(B) \leq 45 \cdot 5 - 28 \cdot 8 = 1$.\\ \\ 
% The inequality above shows that the block $B_{5,c}$ may have a positive contribution, so we need to discuss further to obtain the desired results. There are two possibilities for the structure of exterior faces of $B$ that we need to discuss individually, as followed.
\textbf{Case 1:}
Suppose some $f\in F(G)$ adjacent to $B_{5,c}$ on the outside has length at least $5$.
In this case, $f_{B_{5,c}}\leq 4+\frac{1}{4}\cdot 3+\frac{1}{5}$, $e_{B_{5,c}}=8$, and $g(B_{5,c})\leq -\frac{5}{4}$.

% This implies that there is one exterior face of size at least $5$ joining at least one of the four exterior edges of $B$. Since $B$ has four interior faces, so $\hat{f}(B) \leq 4 + 3 \cdot \frac{1}{4} + \frac{1}{5} = \frac{99}{20}$. Clearly $e(B) = 8$, so $45\hat{f}(B) - 28e(B) \leq 45 \cdot \frac{99}{20} - 28 \cdot 8 = -\frac{5}{4} < 0$.

\textbf{Case 2:}
Suppose every $f\in F(G)$ adjacent to $B_{5,c}$ on the outside is a $4$-face.
We will show that the additional edges constitute these $4$-faces must be placed in a unique way and that they induce trivial triangular blocks.
Let $f\in F(G)$ be the $4$-face adjacent to $x_1x_2$ on the outside.
There are three possibilities as shown in the figure below.
    \begin{center}
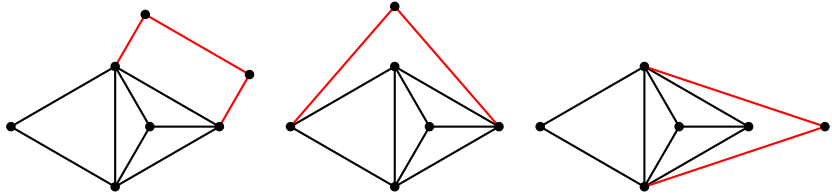
    \begin{tabular}{c c c}
    \begin{tikzpicture}[scale=0.8]
    \draw[thick] (0,1) -- (0,-1);
    \draw[thick] (0,1) -- (-1.732,0);
    \draw[thick] (0,1) -- (1.732,0);
    \draw[thick] (0,-1) -- (-1.732,0);
    \draw[thick] (0,-1) -- (1.732,0);
    \draw[thick] (0,1) -- (0.577,0);
    \draw[thick] (0,-1) -- (0.577,0);
    \draw[thick] (1.732,0) -- (0.577,0);
    \draw[thick, red] (0,1) -- (0.5,1.866);
    \draw[thick, red] (0.5,1.866) -- (2.232,0.866);
    \draw[thick, red] (2.232,0.866) -- (1.732,0);
    \filldraw (0.5,1.866) circle (0.07);
    \filldraw (2.232,0.866) circle (0.07);
    \filldraw (0,1) circle (0.07);
    \filldraw (0,-1) circle (0.07);
    \filldraw (-1.732,0) circle (0.07);
    \filldraw (1.732,0) circle (0.07);
    \filldraw (0.577,0) circle (0.07);
    \end{tikzpicture}
    &
    \begin{tikzpicture}[scale=0.8]
    \draw[thick] (0,1) -- (0,-1);
    \draw[thick] (0,1) -- (-1.732,0);
    \draw[thick] (0,1) -- (1.732,0);
    \draw[thick] (0,-1) -- (-1.732,0);
    \draw[thick] (0,-1) -- (1.732,0);
    \draw[thick] (0,1) -- (0.577,0);
    \draw[thick] (0,-1) -- (0.577,0);
    \draw[thick] (1.732,0) -- (0.577,0);
    \draw[thick, red] (0,2) -- (1.732,0);
    \draw[thick, red] (0,2) -- (-1.732,0);
    \filldraw (0,2) circle (0.07);
    \filldraw (0,1) circle (0.07);
    \filldraw (0,-1) circle (0.07);
    \filldraw (-1.732,0) circle (0.07);
    \filldraw (1.732,0) circle (0.07);
    \filldraw (0.577,0) circle (0.07);
    \end{tikzpicture}
    & 
    \begin{tikzpicture}[scale=0.8]
    \draw[thick] (0,1) -- (0,-1);
    \draw[thick] (0,1) -- (-1.732,0);
    \draw[thick] (0,1) -- (1.732,0);
    \draw[thick] (0,-1) -- (-1.732,0);
    \draw[thick] (0,-1) -- (1.732,0);
    \draw[thick] (0,1) -- (0.577,0);
    \draw[thick] (0,-1) -- (0.577,0);
    \draw[thick] (1.732,0) -- (0.577,0);
    \draw[thick, red] (0,1) -- (3,0);
    \draw[thick, red] (0,-1) -- (3,0);
    \filldraw (3,0) circle (0.07);
    \filldraw (0,1) circle (0.07);
    \filldraw (0,-1) circle (0.07);
    \filldraw (-1.732,0) circle (0.07);
    \filldraw (1.732,0) circle (0.07);
    \filldraw (0.577,0) circle (0.07);
    \end{tikzpicture}
    \end{tabular}
    \end{center}
Since the first two configurations are not $\Theta_6^1$-free, the $4$-face of $G$ adjacent to $x_1x_2$ or $x_2x_3$ on the outside must be placed as in the third configuration.
One can work out a similar argument for $x_1x_4$ or $x_3x_4$ and conclude that if every $f\in F(G)$ adjacent to $B_{5,c}$ on the outside is a $4$-face, then those $4$-faces must be placed as shown below: 
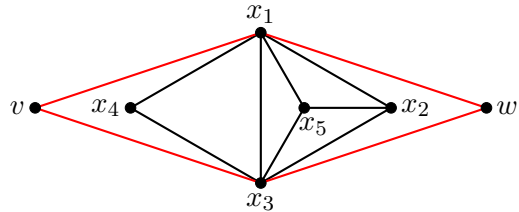
\begin{center}
\begin{tikzpicture}
\draw[thick] (0,1) -- (0,-1);
\draw[thick] (0,1) -- (-1.732,0);
\draw[thick] (0,1) -- (1.732,0);
\draw[thick] (0,-1) -- (-1.732,0);
\draw[thick] (0,-1) -- (1.732,0);
\draw[thick] (0,1) -- (0.577,0);
\draw[thick] (0,-1) -- (0.577,0);
\draw[thick] (1.732,0) -- (0.577,0);
\draw[thick, red] (0,1) -- (-3,0);
\draw[thick, red] (0,-1) -- (-3,0);
\filldraw (-3,0) circle (0.07);
\draw[thick, red] (0,1) -- (3,0);
\draw[thick, red] (0,-1) -- (3,0);
\filldraw (3,0) circle (0.07);
\node[above] at (0,1) {$x_1$};
\node[right] at (1.732,0) {$x_2$};
\node[below] at (0,-1) {$x_3$};
\node[left] at (-1.732,0) {$x_4$};
\node[below] at (0.7,0) {$x_5$};
\node[right] at (3,0) {$w$};
\node[left] at (-3,0) {$v$};
\filldraw (0,1) circle (0.07);
\filldraw (0,-1) circle (0.07);
\filldraw (-1.732,0) circle (0.07);
\filldraw (1.732,0) circle (0.07);
\filldraw (0.577,0) circle (0.07);
\end{tikzpicture}
\end{center}

If $x_1v$ is adjacent to a $3$-cycle (in particular, $3$-face) on the outside, then there are two possibilities as shown in the figure below, and both of them are not $\Theta_6^1$-free. 
Hence, $x_1v$ together with its endpoints forms a trivial triangular block.
The same argument works for $x_3v$, $x_1w$, and $x_3w$.
Since these five blocks always appear together, we define a cluster $\overline{B}$ to be the union of one triangular block $B_{5,c}$ and four trivial triangular blocks $B_2$ who share common $4$-faces with $B_{5,c}$.
Naturally, these $B_2$ blocks will not be considered as clusters individually anymore.
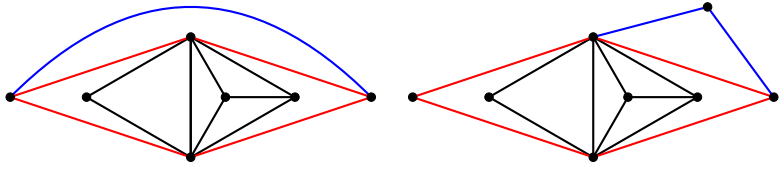
\begin{center}
\begin{tabular}{c c}
\begin{tikzpicture}[scale=0.8]
\draw[thick,blue] (-3,0) .. controls (-1,2) and (1,2) .. (3,0);
\draw[thick] (0,1) -- (0,-1);
\draw[thick] (0,1) -- (0,-1);
\draw[thick] (0,1) -- (-1.732,0);
\draw[thick] (0,1) -- (1.732,0);
\draw[thick] (0,-1) -- (-1.732,0);
\draw[thick] (0,-1) -- (1.732,0);
\draw[thick] (0,1) -- (0.577,0);
\draw[thick] (0,-1) -- (0.577,0);
\draw[thick] (1.732,0) -- (0.577,0);
\draw[thick, red] (0,1) -- (-3,0);
\draw[thick, red] (0,-1) -- (-3,0);
\filldraw (-3,0) circle (0.07);
\draw[thick, red] (0,1) -- (3,0);
\draw[thick, red] (0,-1) -- (3,0);
\filldraw (3,0) circle (0.07);
\filldraw (0,1) circle (0.07);
\filldraw (0,-1) circle (0.07);
\filldraw (-1.732,0) circle (0.07);
\filldraw (1.732,0) circle (0.07);
\filldraw (0.577,0) circle (0.07);
\end{tikzpicture}
& 
\begin{tikzpicture}[scale=0.8]
\draw[thick,blue] (1.9,1.5) -- (3,0);
\draw[thick,blue] (1.9,1.5) -- (0,1);
\draw[thick] (0,1) -- (0,-1);
\draw[thick] (0,1) -- (-1.732,0);
\draw[thick] (0,1) -- (1.732,0);
\draw[thick] (0,-1) -- (-1.732,0);
\draw[thick] (0,-1) -- (1.732,0);
\draw[thick] (0,1) -- (0.577,0);
\draw[thick] (0,-1) -- (0.577,0);
\draw[thick] (1.732,0) -- (0.577,0);
\draw[thick, red] (0,1) -- (-3,0);
\draw[thick, red] (0,-1) -- (-3,0);
\filldraw (-3,0) circle (0.07);
\draw[thick, red] (0,1) -- (3,0);
\draw[thick, red] (0,-1) -- (3,0);
\filldraw (3,0) circle (0.07);
\filldraw (0,1) circle (0.07);
\filldraw (0,-1) circle (0.07);
\filldraw (-1.732,0) circle (0.07);
\filldraw (1.732,0) circle (0.07);
\filldraw (0.577,0) circle (0.07);
\filldraw (1.9,1.5) circle (0.07);
\end{tikzpicture}
\end{tabular}
\end{center}

Let's calculate the contribution by $\overline{B}$.
Since $x_1x_4x_3v$ and $x_1x_2x_3w$ are $4$-faces of $G$, we only need to discuss $f\in F(G)$ adjacent to $x_1v$, $x_3v$, $x_1w$, or $x_3w$ on the outside.
By Remark \ref{adj}, such $f$ has length at least $4$.
Therefore, in this case, $f_{\overline{B}}\leq 6+\frac{1}{4}\cdot 4$, $e_{\overline{B}}=12$, and $g(\overline{B})\leq-21$.

% The arguments and the results for the cases \ref{theta61B5ccase2} and \ref{theta61B5ccase3} are completely the same as case \ref{theta61B5ccase1}. Notice that these are the only cases where we need to calculate the contribution of a cluster of multiple triangular blocks throughout our entire proof, so we can avoid over-counting issue and finish our proof as long as we show that all of these clusters are either the same or edge-disjoint.

To sum up, we partition $\mathcal{B}$ in the following way:
\begin{enumerate}
    \item If all the non-interior faces adjacent to some $B_{5,c}$ block have length $4$, then this $B_{5,c}$ together with the $4$ trivial triangular blocks on its outside form a cluster $\overline{B}$.
    \item Every remaining triangular block in $G$ is a cluster by itself.
\end{enumerate}
To complete the proof, we are left to check that our partition is well-defined, 
that is, we must show that any two $\overline{B}$ clusters are disjoint.
Notice that any other pair of clusters is disjoint by construction.

For the sake of contradiction, suppose two distinct $\overline{B}$ clusters intersect, denoted by $\overline{B_1}$ and $\overline{B_2}$.
If they share a common $B_{5,c}$ block, then, by construction, they are forced to be the same cluster.
Thus, $\overline{B_1}$ and $\overline{B_2}$ must share a common $B_2$ block, denoted by $B_2'$.
Notice that $\overline{B_2}$ has a $3$-cycle which contains $B_2'$ and the diagonal of the $B_{5,c}$ block of $\overline{B_2}$.
Since the $B_{5,c}$ blocks of $\overline{B_1}$ and $\overline{B_2}$ are edge-disjoint, such a $3$-cycle cannot be simultaneously contained in $\overline{B_1}$.
In short, there is a $3$-cycle containing $B_2'$ that is not contained in $\overline{B_1}$, i.e. a $3$-cycle adjacent to $\overline{B_1}$ on the outside of the edge $B_2'$. However, this is a contradiction by the same argument of showing that $x_1v$ induces a trivial triangular block.

We have shown that every cluster has contribution strictly less that zero except for $B_{5,a}$, which may have exactly zero contribution.
By Lemma \ref{tri-b}, we prove the upper bound $ex_{\mathcal{P}}(n,\Theta_6^1)\leq\frac{45}{17}(n-2)$.
It also implies that the equality is achieved only if every triangular block is $B_{5,a}$ adjacent to $5$-faces, as we expected from the extremal constructions.

\section{Proof of Theorem \ref{theta6-2}}
\subsection{Proof of Lower Bound}
Since $\Theta^2_6$ is a particular type of $\Theta_6$, we know a $\Theta_6$-free graph is obviously $\Theta_6^2$-free.
Then this subsection is a corollary of Theorem 5 in \cite{theta6}, which presents an infinite family of $\Theta_6$-free planar graphs satisfying the lower bound.

\subsection{Proof of Upper Bound}
By Proposition \ref{8blocks}, there are only $8$ triangular blocks that are $\Theta^2_6$-free. We can use the same method as in Section \ref{theta6-1} to get the table of contributions below.
\begin{center}
\begin{tabular}{ |c| c | c| c | c| } 
  \hline
  $B$ &Diagram & $f_B \leq$ & $e_B$ & $g(B)=18f_B-11e_B \leq $  \\
  \hline
  \begin{tabular}{l}
       $B_2$
  \end{tabular} & 
  \begin{tabular}{l} \begin{tikzpicture}
    \filldraw (-1,0) circle (0.07);
    \filldraw (1,0) circle (0.07);
    \draw[thick] (-1,0) -- (1,0);
    \end{tikzpicture} \end{tabular} & $\frac{1}{4}\cdot 2$ & $1$ & $-2$ \\
    
  \hline
  \begin{tabular}{l}
       $B_3$
  \end{tabular} &
  \begin{tabular}{l}
  \begin{tikzpicture}
    \filldraw (-1,0) circle (0.07);
    \filldraw (1,0) circle (0.07);
    \filldraw (0,1.732) circle (0.07);
    \draw[thick] (-1,0) -- (1,0);
    \draw[thick] (-1,0) -- (0,1.732);
    \draw[thick] (0,1.732) -- (1,0);
    \end{tikzpicture} \end{tabular} & $1+\frac{1}{4}\cdot 3$ & $3$ & $-\frac{3}{2}$ \\ 
    
  \hline
  \begin{tabular}{l} $B_{4,a}$ \end{tabular}& 
     \begin{tabular}{l}
  \begin{tikzpicture}
    \filldraw (-1,0) circle (0.07);
    \filldraw (1,0) circle (0.07);
    \filldraw (0,1.732) circle (0.07);
    \filldraw (0,0.577) circle (0.07);
    \draw[thick] (-1,0) -- (1,0);
    \draw[thick] (-1,0) -- (0,1.732);
    \draw[thick] (0,1.732) -- (1,0);
    \draw[thick] (-1,0) -- (0,0.577);
    \draw[thick] (1,0) -- (0,0.577);
    \draw[thick] (0,1.732) -- (0,0.577);
    \end{tikzpicture} \end{tabular}& $3+\frac{1}{4}\cdot 2+\frac{1}{6}$& $6$ & \textcolor{red}{$0$} \\
    
    \hline
    \begin{tabular}{l} $B_{4,b}$ \end{tabular}& 
     \begin{tabular}{l}
    \begin{tikzpicture}
    \filldraw (-1,1) circle (0.07);
    \filldraw (1,1) circle (0.07);
    \filldraw (1,-1) circle (0.07);
    \filldraw (-1,-1) circle (0.07);
    \draw[thick] (-1,1) -- (1,1);
    \draw[thick] (-1,1) -- (-1,-1);
    \draw[thick] (1,-1) -- (1,1);
    \draw[thick] (1,-1) -- (-1,-1);
    \draw[thick] (-1,1) -- (1,-1);
    \end{tikzpicture} \end{tabular} & $2+\frac{1}{4}\cdot 4$ & $5$ & $-1$ \\
    
    \hline
    \begin{tabular}{l} $B_{5,a}$ \end{tabular}& 
     \begin{tabular}{l}
    \begin{tikzpicture}
    \filldraw (-1.5,0) circle (0.07);
    \filldraw (1.5,0) circle (0.07);
    \filldraw (0,2.6) circle (0.07);
    \filldraw (0,0.87) circle (0.07);
    \filldraw (0,1.7) circle (0.07);
    \draw[thick] (-1.5,0) -- (1.5,0);
    \draw[thick] (-1.5,0) -- (0,2.6);
    \draw[thick] (0,2.6) -- (1.5,0);
    \draw[thick] (-1.5,0) -- (0,0.87);
    \draw[thick] (1.5,0) -- (0,0.87);
    \draw[thick] (0,2.6) -- (0,0.87);
    \draw[thick] (-1.5,0) -- (0,1.7);
    \draw[thick] (1.5,0) -- (0,1.7);
    \draw[thick] (0,2.6) -- (0,1.7);
    \end{tikzpicture} \end{tabular}& $5+\frac{1}{6}\cdot 3$&$9$ & \textcolor{red}{0}\\

    \hline
    \begin{tabular}{l} $B_{5,b}$ \end{tabular}& 
     \begin{tabular}{l}\begin{tikzpicture}
    \filldraw (-1,1) circle (0.07);
    \filldraw (1,1) circle (0.07);
    \filldraw (1,-1) circle (0.07);
    \filldraw (-1,-1) circle (0.07);
    \filldraw (0,0) circle (0.07);
    \draw[thick] (-1,1) -- (1,1);
    \draw[thick] (-1,1) -- (-1,-1);
    \draw[thick] (1,-1) -- (1,1);
    \draw[thick] (1,-1) -- (-1,-1);
    \draw[thick] (1,1) -- (0,0);
    \draw[thick] (-1,-1) -- (0,0);
    \draw[thick] (-1,1) -- (0,0);
    \draw[thick] (1,-1) -- (0,0);
    \end{tikzpicture}  \end{tabular} & $4+\frac{1}{5}\cdot 4$&8 & $-\frac{8}{5}$\\

    \hline
    \begin{tabular}{l} $B_{5,c}$ \end{tabular}& 
     \begin{tabular}{l}
    \begin{tikzpicture}
    \filldraw (0,1) circle (0.07);
    \filldraw (0,-1) circle (0.07);
    \filldraw (-1.732,0) circle (0.07);
    \filldraw (1.732,0) circle (0.07);
    \filldraw (0.577,0) circle (0.07);
    \draw[thick] (0,1) -- (0,-1);
    \draw[thick] (0,1) -- (-1.732,0);
    \draw[thick] (0,1) -- (1.732,0);
    \draw[thick] (0,-1) -- (-1.732,0);
    \draw[thick] (0,-1) -- (1.732,0);
    \draw[thick] (0,1) -- (0.577,0);
    \draw[thick] (0,-1) -- (0.577,0);
    \draw[thick] (1.732,0) -- (0.577,0);
    \end{tikzpicture} \end{tabular} &$4+\frac{1}{4}\cdot 4$ &$8$ &  \makecell{\textcolor{red}{$2$} \\ \textcolor{red}{(more details in \ref{B5,c theta62})}}\\

    \hline
    \begin{tabular}{l} $B_{5,d}$ \end{tabular}& 
    \begin{tabular}{l}
    \begin{tikzpicture}
    \filldraw (0,1.5) circle (0.07);
    \filldraw (0.951*1.5,0.309*1.5) circle (0.07);
    \filldraw (-0.951*1.5,0.309*1.5) circle (0.07);
    \filldraw (0.588*1.5,-0.809*1.5) circle (0.07);
    \filldraw (-0.588*1.5,-0.809*1.5) circle (0.07);
    \draw[thick] (0.951*1.5,0.309*1.5) -- (0,1.5);
    \draw[thick] (0.951*1.5,0.309*1.5) -- (0.588*1.5,-0.809*1.5);
    \draw[thick] (-0.588*1.5,-0.809*1.5) -- (0.588*1.5,-0.809*1.5);
    \draw[thick] (-0.588*1.5,-0.809*1.5) -- (-0.951*1.5,0.309*1.5);
    \draw[thick] (0,1.5) -- (-0.951*1.5,0.309*1.5);
    \draw[thick] (0,1.5) -- (0.588*1.5,-0.809*1.5);
    \draw[thick] (0,1.5) -- (-0.588*1.5,-0.809*1.5);
    \end{tikzpicture}    \end{tabular} & $3+\frac{1}{4}\cdot 5$& $7$ & $-\frac{1}{2}$\\ 
    \hline
\end{tabular}
\end{center}
\subsubsection{Troublemaker $B_{5,c}$} \label{B5,c theta62}

\begin{center}
\begin{tikzpicture}
\filldraw (0,1) circle (0.07);
\filldraw (0,-1) circle (0.07);
\filldraw (-1.732,0) circle (0.07);
\filldraw (1.732,0) circle (0.07);
\filldraw (0.577,0) circle (0.07);
\draw[thick] (0,1) -- (0,-1);
\draw[thick] (0,1) -- (-1.732,0);
\draw[thick] (0,1) -- (1.732,0);
\draw[thick] (0,-1) -- (-1.732,0);
\draw[thick] (0,-1) -- (1.732,0);
\draw[thick] (0,1) -- (0.577,0);
\draw[thick] (0,-1) -- (0.577,0);
\draw[thick] (1.732,0) -- (0.577,0);
\node[above] at (0,1) {$x_1$};
\node[right] at (1.732,0) {$x_2$};
\node[below] at (0,-1) {$x_3$};
\node[left] at (-1.732,0) {$x_4$};
\node[below] at (0.7,0) {$x_5$};
\end{tikzpicture}
\end{center}

For simplicity, we call the four edges constituting the $4$-face of $B_{5,c}$, that is, $x_1x_2$, $x_2x_3$, $x_3x_4$, and $x_1x_4$ the \textbf{boundaries} of $B_{5,c}$. We first prove the following lemmas. 
\begin{lemma} \label{lemmaoffourfacesofspecialblockintheta62}
If a face adjacent to $B_{5,c}$ is a $4$-face, then it must be one of the following cases.
\end{lemma}
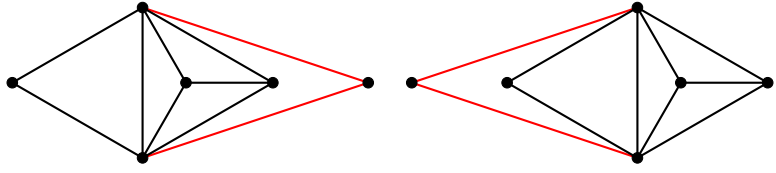
\begin{center}
\begin{tabular}{c c}
\begin{tikzpicture}
\draw[thick, red] (0,1) -- (3,0);
\draw[thick, red] (0,-1) -- (3,0);
\filldraw (0,1) circle (0.07);
\filldraw (0,-1) circle (0.07);
\filldraw (-1.732,0) circle (0.07);
\filldraw (1.732,0) circle (0.07);
\filldraw (0.577,0) circle (0.07);
\draw[thick] (0,1) -- (0,-1);
\draw[thick] (0,1) -- (-1.732,0);
\draw[thick] (0,1) -- (1.732,0);
\draw[thick] (0,-1) -- (-1.732,0);
\draw[thick] (0,-1) -- (1.732,0);
\draw[thick] (0,1) -- (0.577,0);
\draw[thick] (0,-1) -- (0.577,0);
\draw[thick] (1.732,0) -- (0.577,0);
\filldraw (3,0) circle (0.07);
\end{tikzpicture}
&
\begin{tikzpicture}
\draw[thick, red] (0,1) -- (-3,0);
\draw[thick, red] (0,-1) -- (-3,0);
\filldraw (0,1) circle (0.07);
\filldraw (0,-1) circle (0.07);
\filldraw (-1.732,0) circle (0.07);
\filldraw (1.732,0) circle (0.07);
\filldraw (0.577,0) circle (0.07);
\draw[thick] (0,1) -- (0,-1);
\draw[thick] (0,1) -- (-1.732,0);
\draw[thick] (0,1) -- (1.732,0);
\draw[thick] (0,-1) -- (-1.732,0);
\draw[thick] (0,-1) -- (1.732,0);
\draw[thick] (0,1) -- (0.577,0);
\draw[thick] (0,-1) -- (0.577,0);
\draw[thick] (1.732,0) -- (0.577,0);
\filldraw (-3,0) circle (0.07);
\end{tikzpicture}
\end{tabular}
\end{center}
\begin{proof}
    Other than the two cases listed in the lemma, we have the other possible cases of a $4$-face adjacent to a boundary of $B_{5,c}$ as follows. 
    \begin{center}
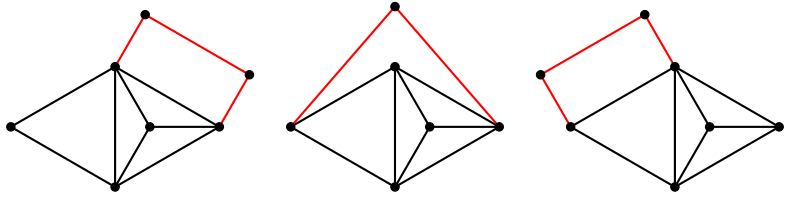
        \begin{tabular}{c c c}
        \begin{tikzpicture}[scale=0.8]
        \draw[thick] (0,1) -- (0,-1);
        \draw[thick] (0,1) -- (-1.732,0);
        \draw[thick] (0,1) -- (1.732,0);
        \draw[thick] (0,-1) -- (-1.732,0);
        \draw[thick] (0,-1) -- (1.732,0);
        \draw[thick] (0,1) -- (0.577,0);
        \draw[thick] (0,-1) -- (0.577,0);
        \draw[thick] (1.732,0) -- (0.577,0);
        \draw[thick, red] (0,1) -- (0.5,1.866);
        \draw[thick, red] (0.5,1.866) -- (2.232,0.866);
        \draw[thick, red] (2.232,0.866) -- (1.732,0);
        \filldraw (0.5,1.866) circle (0.07);
        \filldraw (2.232,0.866) circle (0.07);
        \filldraw (0,1) circle (0.07);
        \filldraw (0,-1) circle (0.07);
        \filldraw (-1.732,0) circle (0.07);
        \filldraw (1.732,0) circle (0.07);
        \filldraw (0.577,0) circle (0.07);
        \end{tikzpicture}
        &
        \begin{tikzpicture}[scale=0.8]
        \draw[thick] (0,1) -- (0,-1);
        \draw[thick] (0,1) -- (-1.732,0);
        \draw[thick] (0,1) -- (1.732,0);
        \draw[thick] (0,-1) -- (-1.732,0);
        \draw[thick] (0,-1) -- (1.732,0);
        \draw[thick] (0,1) -- (0.577,0);
        \draw[thick] (0,-1) -- (0.577,0);
        \draw[thick] (1.732,0) -- (0.577,0);
        \draw[thick, red] (0,2) -- (1.732,0);
        \draw[thick, red] (0,2) -- (-1.732,0);
        \filldraw (0,2) circle (0.07);
        \filldraw (0,1) circle (0.07);
        \filldraw (0,-1) circle (0.07);
        \filldraw (-1.732,0) circle (0.07);
        \filldraw (1.732,0) circle (0.07);
        \filldraw (0.577,0) circle (0.07);
        \end{tikzpicture}
        & 
        \begin{tikzpicture}[scale=0.8]
        \draw[thick] (0,1) -- (0,-1);
        \draw[thick] (0,1) -- (-1.732,0);
        \draw[thick] (0,1) -- (1.732,0);
        \draw[thick] (0,-1) -- (-1.732,0);
        \draw[thick] (0,-1) -- (1.732,0);
        \draw[thick] (0,1) -- (0.577,0);
        \draw[thick] (0,-1) -- (0.577,0);
        \draw[thick] (1.732,0) -- (0.577,0);
        \draw[thick, red] (0,1) -- (-0.5,1.866);
        \draw[thick, red] (-2.232,0.866) -- (-1.732,0);
        \draw[thick, red] (-2.232,0.866) -- (-0.5,1.866);  
        \filldraw (-0.5,1.866) circle (0.07);
        \filldraw (-2.232,0.866) circle (0.07);
        \filldraw (0,1) circle (0.07);
        \filldraw (0,-1) circle (0.07);
        \filldraw (-1.732,0) circle (0.07);
        \filldraw (1.732,0) circle (0.07);
        \filldraw (0.577,0) circle (0.07);
        \end{tikzpicture}
        \end{tabular}
    \end{center}
    These three cases are not $\Theta_6^2$-free.
    % sharing a single edge and sharing two edges $x_1x_4, x_1x_2$ or $x_3x_4, x_3x_2$. In both of these cases, we will form a $\Theta_6^2$.
\end{proof}

\begin{lemma} \label{lemmaofno5cycle}
Suppose exactly two boundaries of $B_{5,c}$ are adjacent to a $4$-face as in Lemma \ref{lemmaoffourfacesofspecialblockintheta62}. 
Then the other two boundaries of $B_{5,c}$ cannot be adjacent to $5$-faces.
\end{lemma}
\begin{proof}
Let a $4$-face join $B_{5,c}$ as in the figure below.
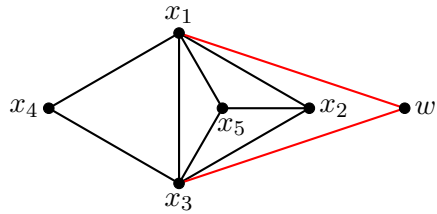
\begin{center}
\begin{tikzpicture}
\draw[thick] (0,1) -- (0,-1);
\draw[thick] (0,1) -- (-1.732,0);
\draw[thick] (0,1) -- (1.732,0);
\draw[thick] (0,-1) -- (-1.732,0);
\draw[thick] (0,-1) -- (1.732,0);
\draw[thick] (0,1) -- (0.577,0);
\draw[thick] (0,-1) -- (0.577,0);
\draw[thick] (1.732,0) -- (0.577,0);
\draw[thick, red] (0,1) -- (3,0);
\draw[thick, red] (0,-1) -- (3,0);
\filldraw (3,0) circle (0.07);
\node[above] at (0,1) {$x_1$};
\node[right] at (1.732,0) {$x_2$};
\node[below] at (0,-1) {$x_3$};
\node[left] at (-1.732,0) {$x_4$};
\node[below] at (0.7,0) {$x_5$};
\node[right] at (3,0) {$w$};
\filldraw (0,1) circle (0.07);
\filldraw (0,-1) circle (0.07);
\filldraw (-1.732,0) circle (0.07);
\filldraw (1.732,0) circle (0.07);
\filldraw (0.577,0) circle (0.07);
\end{tikzpicture}
\end{center}
For the sake of contradiction, suppose that $x_1x_4$ is adjacent to a $5$-face on the outside. 
We know that in a planar graph, there are only two possible $5$-faces: a face with $5$ vertices (namely, a pentagon) or with $4$ vertices (a triangle with one of its vertices connected to an extra edge). 
The second case is impossible, otherwise there exists an additional vertex $y\not\in V(B_{5,c})$ such that $x_1x_4y$ forms a $3$-cycle, and the resultant graph is not $\Theta_6^2$-free by Lemma \ref{poss_tri_block}.
Hence, the $5$-face adjacent to $x_1x_4$ is a pentagon.

We are therefore left with two cases. 

\textbf{Case 1:} Suppose the $5$-face does not join $x_3x_4$.
Then there are three vertices $a$, $b$, and $c$ such that the $5$-face can be written as $x_1x_4abc$, where $a,c \neq x_3$. 
If $b \neq x_3$, then $x_1x_3x_4 \cup x_1x_4abc$ forms a $\Theta_6^2$;
if $b = x_3$, then $x_3ax_4 \cup x_1x_4x_3x_2x_5$ forms a $\Theta_6^2$.
There is always a contradiction, and this case is impossible. 

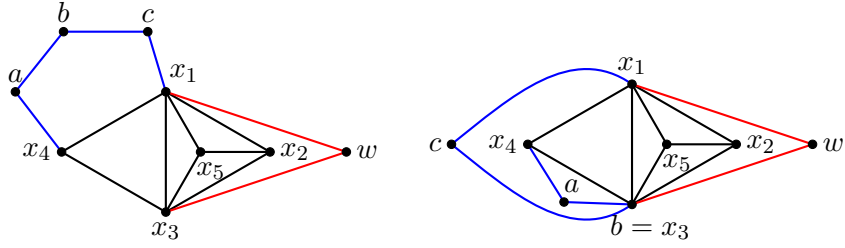
\begin{center}
\begin{tabular}{c c}
\begin{tikzpicture}[scale=0.8]
\draw[thick] (0,1) -- (0,-1);
\draw[thick] (0,1) -- (-1.732,0);
\draw[thick] (0,1) -- (1.732,0);
\draw[thick] (0,-1) -- (-1.732,0);
\draw[thick] (0,-1) -- (1.732,0);
\draw[thick] (0,1) -- (0.577,0);
\draw[thick] (0,-1) -- (0.577,0);
\draw[thick] (1.732,0) -- (0.577,0);
\draw[thick, red] (0,1) -- (3,0);
\draw[thick, red] (0,-1) -- (3,0);
\draw[thick, blue] (-1.732,0) -- (-2.5,1);
\draw[thick, blue] (-1.7,2) -- (-2.5,1);
\draw[thick, blue] (-1.7,2) -- (-0.3,2);
\draw[thick, blue] (0,1) -- (-0.3,2);
\filldraw (-2.5,1) circle (0.07);
\filldraw (-0.3,2) circle (0.07);
\filldraw (-1.7,2) circle (0.07);
\filldraw (3,0) circle (0.07);
\filldraw (0,1) circle (0.07);
\filldraw (0,-1) circle (0.07);
\filldraw (-1.732,0) circle (0.07);
\filldraw (1.732,0) circle (0.07);
\filldraw (0.577,0) circle (0.07);
\node[above] at (0.3,1) {$x_1$};
\node[right] at (1.732,0) {$x_2$};
\node[below] at (0,-1) {$x_3$};
\node[left] at (-1.732,0) {$x_4$};
\node[below] at (0.75,0) {$x_5$};
\node[right] at (3,0) {$w$};
\node[above] at (-2.5,1) {$a$};
\node[above] at (-1.7,2) {$b$};
\node[above] at (-0.3,2) {$c$};
\end{tikzpicture}
&
\begin{tikzpicture}[scale=0.8]
\draw[thick] (0,1) -- (0,-1);
\draw[thick] (0,1) -- (-1.732,0);
\draw[thick] (0,1) -- (1.732,0);
\draw[thick] (0,-1) -- (-1.732,0);
\draw[thick] (0,-1) -- (1.732,0);
\draw[thick] (0,1) -- (0.577,0);
\draw[thick] (0,-1) -- (0.577,0);
\draw[thick] (1.732,0) -- (0.577,0);
\draw[thick, red] (0,1) -- (3,0);
\draw[thick, red] (0,-1) -- (3,0);
\draw[thick, blue] (0,-1) -- (-1.13,-0.96);
\draw[thick, blue] (-1.732,0) -- (-1.13,-0.96);
\draw[thick,blue] (-3,0) .. controls (-2,0.8) and (-1,1.7) .. (0,1);
\draw[thick,blue] (-3,0) .. controls (-2,-0.8) and (-1,-1.7) .. (0,-1);
\filldraw (-1.13,-0.96) circle (0.07);
\filldraw (-3,0) circle (0.07);
\filldraw (3,0) circle (0.07);
\filldraw (0,1) circle (0.07);
\filldraw (0,-1) circle (0.07);
\filldraw (-1.732,0) circle (0.07);
\filldraw (1.732,0) circle (0.07);
\filldraw (0.577,0) circle (0.07);
\node[above] at (0,1) {$x_1$};
\node[right] at (1.732,0) {$x_2$};
\node[below] at (0.3,-1) {$b=x_3$};
\node[left] at (-1.732,0) {$x_4$};
\node[below] at (0.75,0) {$x_5$};
\node[right] at (3,0) {$w$};
\node[above] at (-1,-0.96) {$a$};
\node[left] at (-3,0) {$c$};
\end{tikzpicture}
\end{tabular}
\end{center}

\textbf{Case 2:} Suppose the $5$-face joins $x_3x_4$.
Then there are two vertices $a$ and $b$ such that the $5$-face can be written as $x_1x_4x_3ab$. 
However, $x_1x_2x_5 \cup x_1x_5x_3ab$ forms a $\Theta_6^2$, which is a contradiction. 
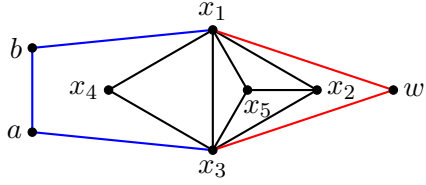
\begin{center}
\begin{tikzpicture}[scale=0.8]
\draw[thick] (0,1) -- (0,-1);
\draw[thick] (0,1) -- (-1.732,0);
\draw[thick] (0,1) -- (1.732,0);
\draw[thick] (0,-1) -- (-1.732,0);
\draw[thick] (0,-1) -- (1.732,0);
\draw[thick] (0,1) -- (0.577,0);
\draw[thick] (0,-1) -- (0.577,0);
\draw[thick] (1.732,0) -- (0.577,0);
\draw[thick, red] (0,1) -- (3,0);
\draw[thick, red] (0,-1) -- (3,0);
\draw[thick, blue] (0,-1) -- (-3,-0.7);
\draw[thick, blue] (-3,0.7) -- (-3,-0.7);
\draw[thick, blue] (-3,0.7) -- (0,1);
\filldraw (-3,0.7) circle (0.07);
\filldraw (-3,-0.7) circle (0.07);
\filldraw (3,0) circle (0.07);
\filldraw (0,1) circle (0.07);
\filldraw (0,-1) circle (0.07);
\filldraw (-1.732,0) circle (0.07);
\filldraw (1.732,0) circle (0.07);
\filldraw (0.577,0) circle (0.07);
\node[above] at (0,1) {$x_1$};
\node[right] at (1.732,0) {$x_2$};
\node[below] at (0,-1) {$x_3$};
\node[left] at (-1.732,0) {$x_4$};
\node[below] at (0.75,0) {$x_5$};
\node[right] at (3,0) {$w$};
\node[left] at (-3,-0.7) {$a$};
\node[left] at (-3,0.7) {$b$};
\end{tikzpicture}
\end{center}
Therefore, $x_1x_4$ cannot be adjacent to a $5$-face. 
By the symmetric argument, $x_3x_4$ cannot be adjacent to a $5$-face either. 
The justification for the second picture of Lemma \ref{lemmaoffourfacesofspecialblockintheta62} adopts a similar idea. 
\end{proof}

%We have the following possible cases of a 5-face adjacent to a boundary of $B_{5,c}$: sharing a single edge, sharing two edges, sharing three vertices and one edge where the edge is $x_4x_1$ or $x_4x_3$, sharing three vertices and one edge that is $x_2x_1$ or $x_2x_3$, and sharing three edges. In all of these cases, the graph is not $\Theta_6^2$-free.

Now we will discuss the contribution of $B_{5,c}$ blocks depending on its adjacent faces.

\textbf{Case 1:} Suppose no boundary of $B_{5,c}$ is adjacent to $4$-faces.
In other words, all boundaries of $B_{5,c}$ are adjacent to faces of length at least $5$. 
In this case,
$f_{B_{5,c}}\leq 4+\frac{1}{5}\cdot 4$, $e_{B_{5,c}}=8$, and $g(B)\leq -\frac{8}{5}$.

\textbf{Case 2:} Suppose one or three of the boundaries of $B_{5,c}$ are adjacent to $4$-faces.
By Lemma \ref{lemmaoffourfacesofspecialblockintheta62}, this case is impossible.

\textbf{Case 3:} Suppose exactly two boundaries of $B_{5,c}$ are adjacent to $4$-faces.
By Lemma \ref{lemmaoffourfacesofspecialblockintheta62}, it is either a $4$-face adjacent to $x_1x_2$ and $x_2x_3$ or a $4$-face adjacent to $x_1x_4$ and $x_3x_4$.
By Lemma \ref{lemmaofno5cycle}, both boundaries that are not adjacent to 4-faces must be adjacent to faces of length at least $6$.
In this case, $f_{B_{5,c}}\leq 4+\frac{1}{4}\cdot 2+\frac{1}{6}\cdot 2$, $e_{B_{5,c}}=8$, and $g(B_{5,c})\leq -1$.

\textbf{Case 4:} Suppose all boundaries of $B_{5,c}$ are adjacent to a $4$-face.
Same as the Case $2$ in Subsection \ref{troublemaker_Theta_61}, we can show that the edges sharing common $4$-faces with this $B_{5,c}$ block must induce trivial triangular blocks. 
Then we can define a cluster $\overline{B}$ in the same way, inducing a well-defined partition of $\mathcal{B}$. For this cluster $\overline{B}$, we have $f_{\overline{B}}\leq 6+\frac{1}{4}\cdot 4$, $e_{\overline{B}}=12$, and $g(\overline{B})\leq-6$.

Since every cluster has contribution at most zero, by Lemma \ref{tri-b}, we prove the upper bound $ex_{\mathcal{P}}(n,\Theta_6^2)\leq\frac{18}{7}(n-2)$.

\section{Acknowledgement}
The research is done as a part of Budapest Semesters in Mathematics, Spring 2024.
Research of Gy\H{o}ri was supported by the National Research, Development and Innovation Office - NKFIH under the grant K 132696. 

\end{document}